\documentclass{amsart}

\usepackage{amsthm,amssymb,latexsym,amsmath}
\usepackage[all]{xy}
\usepackage[english]{babel}
\usepackage[utf8x]{inputenc}
\usepackage{color}

\newcommand{\CC}{{\mathbb C}}\newcommand{\Z}{{\mathbb Z}}
\newcommand{\p}[1]{{\mathbb{P}^{#1}}}

\newcommand{\op}[1]{{\mathcal O}_{\mathbb{P}^{#1}}}

\newcommand{\sing}{\operatorname{Sing}}
\newcommand{\supp}{\operatorname{Supp}}

\newcommand{\calb}{{\mathcal B}}
\newcommand{\calc}{{\mathcal C}}
\newcommand{\cald}{{\mathcal D}}

\newcommand{\calf}{{\mathcal F}}
\newcommand{\calh}{{\mathcal H}}
\newcommand{\cali}{{\mathcal I}}

\newcommand{\call}{{\mathcal L}}
\newcommand{\calm}{{\mathcal M}}
\newcommand{\caln}{{\mathcal N}}
\newcommand{\calo}{{\mathcal O}}
\newcommand{\calp}{{\mathcal P}}

\newcommand{\calr}{{\mathcal R}}
\newcommand{\cals}{{\mathcal S}}
\newcommand{\calt}{{\mathcal T}}
\newcommand{\calz}{{\mathcal Z}}

\newcommand{\inhom}{{\mathcal H}{\it om}}
\newcommand{\inext}{{\mathcal E}{\it xt}}

\DeclareMathOperator{\ext}{Ext}
\DeclareMathOperator{\Hom}{Hom}

\DeclareMathOperator{\Aut}{Aut}
\DeclareMathOperator{\coker}{coker}
\DeclareMathOperator{\im}{im}

\DeclareMathOperator{\rk}{{rk}}
\DeclareMathOperator{\lt}{{length}}

\newcommand{\onto}{\twoheadrightarrow}

\newlength{\rrrr}

\newcommand{\intoo}[1]{\:
\xymatrix@1{\ar@{^(->}[r]^{#1}&}\:}
\newcommand{\ontoo}[1]{\:
\xymatrix@1{\ar@{->>}[r]^{#1}&}\:}

\newtheorem{theorem}{Theorem}

\newtheorem{proposition}[theorem]{Proposition}
\newtheorem{lemma}[theorem]{Lemma}
\newtheorem{corollary}[theorem]{Corollary}

\newtheorem{remark}[theorem]{Remark}

\begin{document}

\title{Two infinite series of moduli spaces of rank 2 sheaves on $\mathbb P^3$}

\author{Marcos Jardim}
\address{IMECC - UNICAMP \\
Departamento de Matem\'atica \\ Rua S\'ergio Buarque de Holanda, 651 \\
13083-970 Campinas-SP, Brazil}
\email{jardim@ime.unicamp.br}

\author{Dimitri Markushevich}
\address{Math\'ematiques -- b\^{a}t.~M2\\
Universit\'e Lille 1\\
F-59655 Villeneuve d'Ascq Cedex, France}
\email{markushe@math.univ-lille1.fr}

\author{Alexander S. Tikhomirov}
\address{Department of Mathematics\\
National Research University,  
Higher School of Economics (HSE)\\
7 Vavilova Street\\ 
117312 Moscow, Russia}
\email{astikhomirov@mail.ru}

\begin{abstract}
We describe new components of the Gieseker--Maruyama moduli scheme $\calm(n)$ of semistable rank 2 sheaves $E$ on $\p3$ with $c_1(E)=0$, $c_2(E)=n$ and $c_3(E)=0$ whose generic point corresponds to non locally free sheaves. We show that such components grow in number as $n$ grows, and discuss how they intersect the instanton component. As an application, we prove that $\calm(2)$ is connected, and identify a connected subscheme of $\calm(3)$ consisting of 7 irreducible components.
\end{abstract}

\maketitle

\section{Introduction}

Let $\calm(c_1;c_2;c_3)$ denote the Gieseker--Maruyama moduli scheme of semistable rank 2 sheaves on $\p3$ with the first, second and third Chern classes equal to $c_1$, $c_2$ and $c_3$, respectively. We will be particularly concerned with $\calm(n):=\calm(0;n;0)$. In addition, we also define $\calb(n)$ to be the open subset of $\calm(n)$ consisting of stable locally free sheaves; and let $\calr(c_1;c_2;c_3)$ denote the open subset of $\calm(c_1;c_2;c_3)$ consisting of stable reflexive sheaves.

The study of stable rank 2 locally free sheaves on $\p3$ in the past 40 years has been mostly concentrated on \emph{instanton bundles}, that is, those stable rank 2 locally free sheaves $E$ on $\p3$ satisfying $c_1(E)=0$ and $h^1(E(-2))=0$. Let $\cali(n)$ denote the moduli space of instanton bundles $E$ with $c_2(E)=n$, regarded as an open subset of $\calm(n)$; the basic questions about its geometry have been settled just recently: it is an irreducible \cite{T1,T2}, nonsingular \cite{JV}, affine \cite{CO} variety of dimension $8n-3$. However, $\overline{\cali(n)}$, the closure of $\cali(n)$ within $\calm(n)$, is not the only irreducuble component of $\calm(n)$ for $n\ge2$; in fact, Ein showed in \cite{Ein} that $\calb(n)$ has several irreducible components as soon as $n\ge3$ and that the number of irreducible components of $\calb(n)$ is not bounded as $n$ grows.

In addition, the closure $\overline{\calb(n)}$ of $\calb(n)$ within $\calm(n)$ does not exhaust $\calm(n)$ already for $n\ge2$, as it was observed by Le Potier \cite[Chapter 7]{LeP} and Trautmann \cite{Tr}. In other words, for each $n\ge2$, $\calm(n)$ possesses entire irreducible components whose generic point corresponds to a stable rank 2 torsion free sheaf which is not locally free. Such components are the main focus of the present paper.

To be more precise, let $E$ be a rank 2 torsion free sheaf on $\p3$ with $c_1(E)=0$, $c_2(E)=n$ and $c_3(E)=0$. Clearly, $c_1(E^{\vee\vee})=0$; we denote $m:=c_2(E^{\vee\vee})$ and $l=c_3(E^{\vee\vee})/2$. Setting $Q_E:=E^{\vee\vee}/E$, one has the fundamental sequence 
\begin{equation}\label{ddsqc}
0 \to E \to E^{\vee\vee} \to Q_E \to 0
\end{equation}
from which one can check that $c_2(Q_E)=-(n-m)$ and $c_3(Q_E)=2l$. If $E$ is not locally free, then $Q_E\ne0$ and there are three possibilities:
\begin{itemize}
\item[(i)] $\dim Q_E=0$; in this case, $n=m$ and $E^{\vee\vee}$ is not locally free; we say that $E$ \emph{has 0-dimensional singularities};
\item[(ii)] $Q_E$ has pure dimension 1; in this case, $n>m$ and we say that $E$ \emph{has 1-dimensional singularities};
\item[(iii)] $\dim Q_E=1$, but it contains 0-dimensional subsheaves; in this case, we say that $E$ \emph{has mixed singularities}.
\end{itemize}
Note that in general $\supp(Q_E)\subseteq\sing(E)$, with equality if $E^{\vee\vee}$ is locally free. Remark that $\sing(E)$ may contain 0-dimensional components even when $Q_E$ has pure dimension one.

We present a systematic construction of irreducible components of $\calm(n)$ whose generic point corresponds to stable rank 2 torsion free sheaves with 0- and 1-dimensional singularities, see Theorem \ref{0d comp's} and Theorem \ref{Thm 4.6} below, respectively. Furthermore, we also show that the number of such components also grows as $n$ grows, cf. Theorem \ref{ein 0d}, for the 0-dimensional case, and Theorem \ref{Thm 8.7}, for the 1-dimensional case, below. 

These results raise the questions of whether it is possible to enumerate all of the irreducible components of $\calm(n)$, at least for low values of $n$, and whether $\calm(n)$ is connected. Indeed, it is not difficult to check that $\calm(1)$ is irreducible (see Section \ref{M(2)}), while Le Potier \cite{LeP} and Trautmann \cite{Tr} showed that $\calm(2)$ has precisely three irreducible components, $\overline{\cali(2)}$ plus two additional ones. In Section \ref{M(2)}, we show, in addition, that the generic point of each of the two so-called \emph{Trautmann components} identified by Le Potier corresponds to a sheaf with 0-dimensional singularities, and that $\calm(2)$ is connected. 

Finally, we show in Section \ref{M(3)} that $\calm(3)$ has at least seven irreducible components. In addition, we provide a discussion on how these various components intersect each other, showing that their union forms a connected subscheme of $\calm(3)$. 

\bigskip

\noindent {\bf Acknowledgements.}
MJ is partially supported by the CNPq grant number 303332/2014-0 and the FAPESP grant number 2014/14743-8. 
DM was partially supported by Labex CEMPI  (ANR-11-LABX-0007-01).
Part of this paper was written while AST visited the University of Campinas under the FAPESP grant number 2014/22807-6. AST also acknowledges the support from the Max Planck Institute in Bonn, where some of the main ideas of this work were conceived during the winter of 2014.

\section{Torsion free sheaves with 0-dimensional singularities} \label{section:0d}

Let us begin by fixing some basic facts about torsion free sheaves $E$ with 0-dimensional singularities. Given any coherent sheaf $G$ on $\p3$, one has $\inext^3(E,G)=0$ and $\inext^2(E,G)\simeq \inext^3(Q_E,G)$ due to the reflexivity of $E^{\vee\vee}$. It follows that torsion free sheaves $E$ with 0-dimensional singularities have homological dimension equal to 2; in other words, $E$ admits a resolution of the form
\begin{equation}\label{E resolution}
0 \to L_2 \to L_1 \to L_0 \to E \to 0
\end{equation}
with each $L_k$ ($k=0,1,2$) being a locally free sheaf.

Note that $\inext^1(E,E)$ and $\inext^2(E,E)$ are $0$-dimensional sheaves, while $\inext^3(E,E)$ vanishes. Thus using the spectral sequence of  local-to-global ext's, we obtain: \label{l2g ext}
\begin{itemize}
\item[(i)] $\ext^1(E,E) = H^1(\inhom(E,E)) \oplus \ker d^{01}_2$
\item[(ii)] $\ext^2(E,E) = \ker d^{02}_3 \oplus \coker d^{01}_2$
\item[(iii)] $\ext^3(E,E) = \coker d^{02}_3$
\end{itemize}
where $d^{01}_2$ and $d^{02}_3$ are the spectral sequence maps
\begin{equation}\label{d012} d^{01}_2 : H^0(\inext^1(E,E)) \to H^2(\inhom(E,E)) ~~ {\rm and} \end{equation}
\begin{equation}\label{d023} d^{02}_3 : H^0(\inext^2(E,E)) \to H^3(\inhom(E,E)). \end{equation}
It then follows that 
\begin{equation}\label{chi ext tf 1}
\sum_{j=0}^3 (-1)^j\dim\ext^j(E,E) = \chi(\inhom(E,E)) - h^0(\inext^1(E,E)) + h^0(\inext^2(E,E)).
\end{equation}

\begin{remark} \rm 
Observe that for a reflexive sheaf $F$ (so that $\inext^2(F,F)=0$) the previous expressions for $\ext^j(F,F)$ simplify to
\begin{itemize}
\item[(i')] $\ext^1(F,F) = H^1(\inhom(F,F)) \oplus \ker d^{01}_2$;
\item[(ii')] $\ext^2(F,F) = \coker d^{01}_2$;
\item[(iii')] $\ext^3(F,F) = H^3(\inhom(F,F))$,
\end{itemize}
where $d^{01}_2$ is the spectral sequence map $d^{01}_2 : H^0(\inext^1(F,F)) \to H^2(\inhom(F,F))$.
Note as well that (\ref{chi ext tf 1}) simplifies to 
\begin{equation}\label{chi ext refl}
\sum_{j=0}^3 (-1)^j\dim\ext^j(F,F) = \chi(\inhom(F,F)) - h^0(\inext^1(F,F)) .
\end{equation}  
\end{remark}

\medskip

\begin{lemma}\label{chi ext tf lemma}
If $E$ is a rank 2 torsion free sheaf with 0-dimensional singularities and $c_1(E)=0$, then
$$ \sum_{j=0}^3 (-1)^j\dim\ext^j(E,E) = - 8c_2(E) + 4. $$
\end{lemma}
\begin{proof}
The strategy is to show that
$$ \chi(\inhom(E,E)) - h^0(\inext^1(E,E)) + h^0(\inext^2(E,E)) = $$
$$ = \chi(\inhom(E^{\vee\vee},E^{\vee\vee})) - h^0(\inext^1(E^{\vee\vee},E^{\vee\vee})). $$
The desired equality will follow from (\ref{chi ext tf 1}), (\ref{chi ext refl}), and
$$ \sum_{j=0}^3 (-1)^j\dim\ext^j(E^{\vee\vee},E^{\vee\vee}) = - 8c_2(E^{\vee\vee}) + 4, $$
see \cite[Prop. 3.4]{H-r}.

Indeed, applying the functor $\inhom(\cdot,E)$ to the fundamental sequence (\ref{ddsqc}) we obtain the isomorphism
$\inext^2(E,E)\simeq\inext^3(Q_E,E)$ plus the exact sequence
\begin{align}
0 \to \inhom(E^{\vee\vee},E) & \to \inhom(E,E) \to \inext^1(Q_E,E) \to \inext^1(E^{\vee\vee},E) \to \label{inhom-e} \\
& \to \inext^1(E,E) \to \inext^2(Q_E,E) \to 0,
\end{align}
since $\inext^2(E^{\vee\vee},E)=0$ because $E^{\vee\vee}$ is reflexive. Next, apply the functor \linebreak $\inhom(E^{\vee\vee},\cdot)$ to the fundamental sequence (\ref{ddsqc}), obtaining
\begin{align} 
0 \to \inhom(E^{\vee\vee},E) & \to \inhom(E^{\vee\vee},E^{\vee\vee})  \to 
\inhom(E^{\vee\vee},Q_E) \to \label{inhomf-} \\
& \to \inext^1(E^{\vee\vee},E) \to \inext^1(E^{\vee\vee},E^{\vee\vee}) \to \inext^1(E^{\vee\vee},Q_E) \to 0.
\end{align}
Comparing Euler characteristics of these last two sequences, we conclude that
$$ \chi(\inhom(E^{\vee\vee},E^{\vee\vee})) - \chi(\inext^1(E^{\vee\vee},E^{\vee\vee})) = $$
$$ = \chi(\inhom(E,E)) - \chi(\inext^1(E,E)) - \chi(\inext^1(Q_E,E)) + \chi(\inext^2(Q_E,E)) + $$
$$ +\chi(\inhom(E^{\vee\vee},Q_E)) - \chi(\inext^1(E^{\vee\vee},Q_E)). $$
Thus, since $\chi(\inext^2(E,E))=\chi(\inext^3(Q_E,E))$, it is now enough to show that
\begin{equation} \label{sum=sum}
\sum_{j=0}^3 (-1)^j\chi(\inext^j(Q_E,E)) = - \sum_{j=0}^3 (-1)^j\chi(\inext^j(E^{\vee\vee},Q_E)),
\end{equation}
noticing that $\inhom(Q_E,E)=0$ and $\inext^j(E^{\vee\vee},Q_E)=0$ for $j=2,3$.

We first consider the left hand side of (\ref{sum=sum}). One can break a locally free resolution of $E$ as in (\ref{E resolution}) into short exact sequences
$$ 0 \to L_2 \to L_1 \to T \to 0 ~~{\rm and}~~ 0 \to T \to L_0 \to E \to 0 . $$
Applying the functor $\inhom(Q_E,\cdot)$ to the first sequence we obtain
$$ 0 \to \inext^2(Q_E,T) \to \inext^3(Q_E,L_2) \to \inext^3(Q_E,L_1) \to \inext^3(Q_E,T) \to 0,$$
with all the other sheaves vanishing. Passing to Euler characteristics we obtain
$$ \chi(\inext^2(Q_E,T)) - \chi(\inext^3(Q_E,T)) = \chi(\inext^3(Q_E,L_2)) - \chi(\inext^3(Q_E,L_1)). $$
But $\inext^3(Q_E,L_k)=\inext^3(Q_E,\op3)\otimes L_k$, hence $\chi(\inext^3(Q_E,L_k))=\rk(L_k)\cdot\chi(Q_E)$. Therefore
\begin{equation}\label{chi1}
\chi(\inext^2(Q_E,T)) - \chi(\inext^3(Q_E,T)) = (\rk(L_2)-\rk(L_1))\chi(Q_E).
\end{equation}
Next, apply the functor $\inhom(Q_E,\cdot)$ to the second part of (\ref{E resolution}) to obtain the isomorphism $\inext^1(Q_E,E)\simeq \inext^2(Q_E,T)$ and the exact sequence
$$ 0 \to \inext^2(Q_E,E) \to \inext^3(Q_E,T) \to \inext^3(Q_E,L_0) \to \inext^3(Q_E,E) \to 0. $$
Passing to Euler characteristics we obtain
$$ \chi(\inext^2(Q_E,E)) - \chi(\inext^3(Q_E,E)) = \chi(\inext^3(Q_E,T)) - \chi(\inext^3(Q_E,L_0)); $$
Subtracting $\inext^1(Q_E,E)$ from the left hand side and $\inext^2(Q_E,T)$ from the right hand side, and then substituting for (\ref{chi1}) we obtain
$$ \sum_{j=0}^3 (-1)^j\chi(\inext^j(Q_E,E)) = (\rk(L_1)-\rk(L_2)-\rk(L_0))\cdot\chi(Q_E) = -2\chi(Q_E) .$$

Finally, we compute the right hand side of (\ref{sum=sum}) in a similar way. Take a locally free resolution of $E^{\vee\vee}$:
$$ 0 \to M_1 \to M_0 \to E^{\vee\vee} \to 0. $$
Applying the functor $\inhom(\cdot,Q_E)$ and passing to Euler characteristics we obtain
\begin{align*}
\chi(\inhom(E^{\vee\vee},Q_E)) - \chi(\inext^1(E^{\vee\vee},Q_E)) & = 
\chi(\inhom(M_0,Q_E)) - \chi(\inhom(M_1,Q_E)) \\ & = 2\chi(Q_E),
\end{align*}
as desired.
\end{proof}

To complete this section, we consider semistable rank 2 torsion free sheaves with 0-dimensional singularities. 

\begin{lemma}\label{stability-lemma}
Let $E$ be a rank 2 torsion free sheaf on $\p3$ with $c_1(E)=0$, $c_2(E)=n$, $c_3(E)=0$, and with 0-dimensional singularities. If $E$ is semistable, then $E^{\vee\vee}$ is stable.
\end{lemma}

We remark that the vanishing of the third Chern class is an essential hypothesis: the sum of the ideal sheaves $I_x\oplus I_y$ of two points $x,y\in\p3$ is semistable and with 0-dimensional singularities, but $(I_x\oplus I_y)^{\vee\vee}$
is not stable.

\begin{proof}
If $E^{\vee\vee}$ is not $\mu$-stable (or, equivalently, stable), then it has a section $\sigma$. We can then form the following diagram
\begin{equation}\label{stability-diag}\begin{gathered}
\xymatrix{
         & 0 \ar[d]              & 0 \ar[d]                      & 0 \ar[d]                   & \\
0 \ar[r] & I_\Delta \ar[d]\ar[r] & \op3 \ar[d]^\sigma \ar[r]     & \calo_\Delta \ar[d] \ar[r] & 0 \\
0 \ar[r] & E \ar[r]              & E^{\vee\vee} \ar[r]^{\varphi} & Q_E \ar[r]                 & 0
} \end{gathered}\end{equation}
where $\Delta$ is a 0-dimensional scheme contained in the support of $Q_E$, and $I_\Delta$ is its ideal sheaf. Notice that one cannot have $\varphi\sigma=0$ because $h^0(E)=0$ by semistability. 

Let $d$ denote the length of $\Delta$; it follows that
$$ p_E(k) - p_{I_\Delta}(k) = -(k+2)n/2 - d < 0 ~~{\rm for} ~~ k ~~ {\rm sufficiently~large} $$
thus $I_\Delta$ would destabilize $E$, contradicting our hypothesis.
\end{proof}

It follows that the only properly semistable torsion free sheaf $E$ with $c_1(E)=c_3(E)=0$, and with 0-dimensional singularities is $2\cdot\op3$. Indeed, assume that $Q_E\ne0$; if $E$ is semistable, then $E^{\vee\vee}$ is $\mu$-stable by Lemma \ref{stability-lemma} above, hence $E$ is also $\mu$-stable and thus stable. When $Q_E=0$, this claim is just \cite[Remark 3.1.1]{H-r}.  

In addition, Hartshorne provides in \cite[Thm. 8.2(b)]{H-r} a bound for the third Chern class of a stable rank 2 reflexive sheaf on $\p3$. Translating this bound to our context, we have the following statement.

\begin{corollary}\label{bound l'}
If $E$ is a semistable rank 2 sheaf on $\p3$ with $c_1(E)=0$, $c_2(E)=n$, $c_3(E)=0$, and with 0-dimensional singularities, then $c_3(E^{\vee\vee})\le n^2-n+2$.
\end{corollary}

\begin{lemma}\label{ext lemma}
Let $E$ be a rank 2 torsion free sheaf on $\p3$ with $c_1(E)=0$, $c_2(E)=n$, $c_3(E)=0$, and with 0-dimensional singularities. If $E$ is stable, then
\begin{itemize}
\item[(a)] $\dim\ext^1(E,E) = 8n-3 + \dim\ext^2(E,E)$;
\item[(b)] $\ext^1(E,E) = H^1(\inhom(E,E)) \oplus \ker d^{01}_2$;
\item[(c)] $\ext^2(E,E) = H^0(\inext^2(E,E)) \oplus \coker d^{01}_2$.
\end{itemize}
\end{lemma}
\begin{proof}
The stability of $E$ implies that $h^0(\inhom(E,E))=1$ and $\ext^3(E,E)=0$. Thus items (a) and (b) are immediate from items (i)-(iii) in page \pageref{l2g ext}. 

Since $\inext^1(Q_E,E)$ has dimension zero, we get from sequence (\ref{inhom-e}) that 
\begin{equation}\label{iso inhoms}
H^i(\inhom(E,E))\simeq H^i(\inhom(E^{\vee\vee},E)) ~~{\rm for}~~ i=2,3.
\end{equation}
Similarly, since $\inext^1(E^{\vee\vee},E)$ has dimension zero, we get from sequence (\ref{inhomf-}) that 
$$ H^i(\inhom(E^{\vee\vee},E^{\vee\vee}))\simeq H^i(\inhom(E^{\vee\vee},E)) ~~{\rm for}~~ i=2,3. $$ 
Putting the isomorphisms above, we get
\begin{equation}\label{iso hi inhom}
H^i(\inhom(E,E))\simeq H^i(\inhom(E^{\vee\vee},E^{\vee\vee})) ~~{\rm for}~~ i=2,3.
\end{equation}
In particular, $H^3(\inhom(E,E))\simeq\ext^3(E^{\vee\vee},E^{\vee\vee})=0$ since $E^{\vee\vee}$ is stable. Thus the spectral sequence map $d^{02}_3$ in (\ref{d023}) vanishes. This gives item (c) of the Lemma. 
\end{proof}

\subsection{Components of sheaves with 0-dimensional singularities} \label{0dsing}

In this section we will show how to produce irreducible components of $\calm(n)$ whose generic point corresponds to a sheaf with 0-dimensional singularities.

Start by considering the following ingredients:
\begin{itemize}
\item[(i)] a stable rank 2 reflexive sheaf $F$ on $\p3$ with $c_1(F)=0$, $c_2(F)=n$ and $c_3(F)=2l$;
\item[(ii)] a 0-dimensional sheaf $Q$ of length $l$ on $\p3$;
\item[(iii)] an epimorphism $\varphi: F \to Q$.
\end{itemize}
Now let $E:=\ker\varphi$. Clearly, this is a ($\mu$-)stable rank 2 torsion free sheaf with $c_1(E)=0$, $c_2(E)=n$ and $c_3(E)=0$ such that $E^{\vee\vee}=F$ and $E^{\vee\vee}/E=Q$; in particular, $\sing(E)$ has dimension 0.

\begin{proposition}\label{4l}
Let $F$ be a stable rank 2 reflexive sheaf with $c_1(F)=0$, $c_2(F)=n$ and $c_3(F)=2l$ such that $\ext^2(F,F)=0$. Take $l$ distinct points $q_1,\dots,q_l$ such that $\{q_1,\dots,q_l\}\cap\sing(F)=\emptyset$, and set $Q:=\oplus_{j=1}^{l}\calo_{q_j}$. Then the kernel $E$ of any epimorphism $\varphi: F \onto Q$ satisfies $\dim\ext^1(E,E)=8n-3+4l$.
\end{proposition}

\begin{proof}
Since $E$ is stable, it is enough, by Lemma \ref{ext lemma}(a), that $\dim\ext^2(E,E)=4l$. The first step is to show that the spectral sequence map (\ref{d012}) is surjective. Indeed, one has the commutative diagram
\begin{equation} \label{d012-diagram}
\xymatrix{ H^0\inext^1(F,E) \ar[d]\ar[r]^{d^{01}_2} & H^2(\inhom(F,E)) \ar[d]^{\simeq} \\
H^0\inext^1(E,E) \ar[r]^{d^{01}_2} & H^2(\inhom(E,E))
}\end{equation}
where vertical arrow in the left is the natural map coming from the exact sequence 
\begin{equation}\label{efq}
0\to E \to F \to Q \to 0, 
\end{equation}
while the vertical arrow in the right is the natural isomorphism obtained as in (\ref{iso inhoms}). Applying
$\Hom(F,\cdot)$ to the sequence (\ref{efq}) we get
$$ \ext^1(F,Q) \to \ext^2(F,E) \to \ext^2(F,F). $$

To see that $\ext^1(F,Q)=0$, note that $H^i(\inext^j(F,Q))=0$ if $i,j\ne0$: indeed, $\inext^j(F,Q)=0$ for $j=2,3$ because $F$ is reflexive; $\inext^1(F,Q)=0$ because the singular locus of $F$ is disjoint from the support of $Q$; and $\inhom(F,Q)$ has dimension 0. It follows from the spectral sequence of local to global Ext's that $\ext^j(F,Q)=0$ for $j>0$. Since, by hypothesis, $\ext^2(F,F)=0$, it follows that $\ext^2(F,E)$. Since $\ext^2(F,E)=\coker d^{01}_2$, we obtain that the top horizontal map is surjective, hence the bottom horizontal map, which is precisely the spectral sequence map (\ref{d012}), is also surjective.

It then follows from Lemma \ref{ext lemma}(c) that $\dim\ext^2(E,E)=h^0(\inext^2(E,E))$. To compute this, note that
\begin{equation}\label{sum of exts}
H^0(\inext^2(E,E)) = \bigoplus_{p\in\sing(E)} \ext^2_{\calo_{\p3,p}}(E_p,E_p),
\end{equation}
where $\sing(E)=\sing(F)\cup\{q_1,\dots,q_l\}$.

First, take $p\in\sing(F)$; since $p\notin \supp(Q)$, we get $E_p\simeq F_p$. However,
$\inext^2_{\calo_{\p3,p}}(F_p,F_p)=\inext^2(F,F)_p=0$ because $F$, being reflexive, has cohomological dimension 1.

Next, take $p=q_j$ for some $1\le j\le l$; restricting the sequence (\ref{efq}) to an open affine subset $U$ of $\p3$ containing $p$ but none of the other singularities of $F$, we have the following short exact sequence of sheaves on $U$:
$$ 0 \to \calo_{U} \oplus I_{p/U} \to 2\cdot\calo_{U} \to \calo_{p/U} \to 0, $$
where $I_{p/U}$ denotes the ideal sheaf of the point $p\in U$ and $\calo_{p/U}$ denotes the structure sheaf of the point $p$ as a subscheme of $U$. It follows that
$$ \ext^2_{\calo_{\p3,p}}(E_p,E_p) = H^0(\inext^2_{\calo_{U}}(I_{p/U},\calo_{p/U})) \oplus 
H^0(\inext^2_{\calo_{U}}(I_{p/U},I_{p/U})). $$
We argue that the first summand has length 1, while the second one has length 3. Indeed, we might as well perform the calculation globally, using the ideal sheaf $I_{p/\p3}$ of the point $p\in\p3$ and its structure sheaf $\calo_{p/\p3}$.

From the exact sequence
$$ 0 \to I_{p/\p3} \to \op3 \to \calo_{p/\p3} \to 0 $$
we obtain that $\inext^2(I_{p/\p3},\op3)\simeq \inext^3(\calo_{p/\p3},\op3)\simeq \calo_{p/\p3}$, so it has length 1.

Now use the resolution of $I_{p/\p3}$ by locally free sheaves:
$$ 0\to \op3(-3) \to 3\cdot\op3(-2) \to 3\cdot\op3(-1) \to I_{p/\p3} \to 0 .$$
Applying the functor $\inhom(\cdot,I_{p/\p3})$ to the sequence
$$ 0 \to G \to 3\cdot\op3(-1) \to I_{p/\p3} \to 0 $$
we obtain that $\inext^2(I_{p/\p3},I_{p/\p3})\simeq\inext^1(G,I_{p/\p3})$. Applying the same functor to the exact sequence   
$$ 0\to \op3(-3) \to 3\cdot\op3(-2) \to G \to 0 $$
we obtain the sequence
$$ 0 \to 3\cdot I_{p/\p3}(2) \to I_{p/\p3}(3) \to \inext^1(G,I_{p/\p3}) \to 0 .$$
Note that the cokernel of the first arrow is just $I_{p/\p3}\otimes\calo_{p/\p3}\simeq 3\cdot\calo_{p/\p3}$, so it has length 3.

Thus the points of $\sing(F)$ do not contribute to (\ref{sum of exts}), while each of the $l$ points in $\supp(Q)$ contributes with a sheaf of length 4. We conclude that $\dim\ext^2(E,E)=4l$, as desired. 
\end{proof}

\bigskip

Now let $\cals(n,l)$ denote an irreducible, open subset of $\calr(0;n;2l)$ whose points correspond to stable reflexive sheaves $F$ satisfying $\ext^2(F,F)=0$; in particular, $\cals(n,l)$ must be the nonsingular locus of an irreducible component of $\calr(0;n;2l)$ of expected dimension $8n-3$. In the product $\cals(n,l)\times(\p3)^l$, we consider the open subset
$$ \left( \cals(n,l)\times(\p3)^l \right)^0 := 
\left\{( [F],q_1,\dots,q_l) ~|~ q_i\ne q_j ~,~ q_i\notin\sing(F) \right\} .$$
Clearly, a point in $\left( \cals(n,l)\times(\p3)^l \right)^0$ can be regarded as a pair of sheaves $([F],Q:=\oplus_{j=1}^{l}\calo_{q_j})$ which fulfills the condition of Proposition \ref{4l}. 
Next, with $([F],Q)\in \left( \cals(n,l)\times(\p3)^l \right)^0$ as above, consider the open set $\Hom(F,Q)_e$ of $\Hom(F,Q)$ consisting of epimorphisms $\varphi: F\twoheadrightarrow Q$; the group $\Aut(Q)$ of automorphisms of the sheaf $Q$ acts on $\Hom(F,Q)_e$ just by homotheties on each factor $\calo_{q_j}$ of $Q$. 

Putting all these data together, we construct the set of triples
$$ \calt(n,l) = \left\{ ([F],Q,\varphi) ~|~ ([F],Q)\in \left( \cals(n,l)\times(\p3)^l \right)^0 ~,~
\varphi\in \Hom(F,Q)_e/\Aut(Q) \right\}. $$ 
By construction, $\calt(n,l)$ is an irreducible, quasi-projective variety of dimension $8n-3+4l$. Indeed, one has the surjective projection
$$  \calt(n,l) \rightarrow \left( \cals(n,l)\times(\p3)^l \right)^0 ~~, ~~ ([F],Q,\varphi) \mapsto ([F],Q) $$
onto an irreducible base variety of dimension $8n-3+3l$, with fibers given by 
$$ \Hom(F,Q)_e/\Aut(Q) \overset{\textrm{open}}{\hookrightarrow} \Hom(F,Q)/\Aut(Q) $$
which have dimension $2l-l=l$. 

To each point $\mathbf{t}:=([F],Q,\varphi)\in\calt(n,l)$ one associates the sheaf
$$ E(\mathbf{t}):= \ker\{ \varphi:F\twoheadrightarrow Q \} $$
which defines a point $[E(\mathbf{t})]$ in $\calm(n)$. Proposition \ref{4l} tell us that, for each
$\mathbf{t}\in\calt(n,l)$, 
$$ \dim\ext^1(E(\mathbf{t}),E(\mathbf{t}))=\dim \calt(n,l), $$
therefore the image of $\calt(n,l)$ into $\calm(n)$ is a dense open subset of an irreducible component of $\calm(n)$; to simplify notation, we denote such component by $\overline{\calt(n,l)}$, the closure of the image of $\calt(n,l)$ within $\calm(n)$.

We summarize the considerations above into the following result.

\begin{theorem} \label{0d comp's}
For every nonsingular irreducible component $\calf$ of $\calr(0;n;2l)$ of expected dimension $8n-3$, there exists an irreducible component $\overline{\calt(n,l)}$ of dimension $8n-3+4l$ in $\calm(n)$ whose generic point $[E]$ satisfies $[E^{\vee\vee}]\in\calf$ and $\lt(Q_E)=l$.
\end{theorem}

\subsection{An Ein type result for sheaves with 0-dimensional singularities}

Recall that Ein has shown in \cite[Proposition 3.6]{Ein} the the number of irreducible components of $\calb(n)$ is unbounded as $n$ grows. We now prove a similar statement for those irreducible components of $\calm(n)$ whose generic points correspond to sheaves with 0-dimensional singularities.

We begin by considering morphisms
$$ \alpha: a\cdot\op3(-3) \oplus b\cdot\op3(-2) \oplus c\cdot\op3(-1) \to (a+b+c+2)\cdot\op3 $$
whose degeneracy locus
$$ \Delta(\alpha) = \{ x\in\p3 ~|~ \alpha(x) ~~ \text{is not injective} \} $$
is 0-dimensional. It follows that the cokernel of $\alpha$ is a rank 2 reflexive sheaf on $\p3$, which we normalize as to fit into the short exact sequence:
\begin{equation}\label{newclass}
0 \to a\cdot\op3(-3) \oplus b\cdot\op3(-2) \oplus c\cdot\op3(-1) \stackrel{\alpha}{\longrightarrow}
(a+b+c+2)\cdot\op3 \to F(k) \to 0,
\end{equation} 
with $a,b,c\ge0$ and such that $3a+2b+c$ is nonzero and even; we set $k:=(3a+2b+c)/2$, so that $c_1(F)=0$.

For simplicity of notation, let
$$ G_{(a,b,c)}:=a\cdot\op3(-3) \oplus b\cdot\op3(-2) \oplus c\cdot\op3(-1). $$
The dimension of the family of rank 2 reflexive sheaves constructed as in equation (\ref{newclass}) is given by
$$ \dim\Hom\left(G_{(a,b,c)}, (a+b+c+2)\cdot\op3\right) - \dim\Aut\left(G_{(a,b,c)}\right) - (a+b+c+2)^2 + 1 = $$ 
$$ = 8k^2+24k-8(b+c)-3 = 8c_2(F) - 3. $$

One easily checks for that $h^0(F)=0$ for every $F$ given by (\ref{newclass}), thus $F$ is always stable. In addition, it is not hard to check that $\ext^2(F,F)=0$. Indeed, applying the functor $\Hom(\cdot,F(k))$ to the sequence (\ref{newclass}), we obtain
$$ \ext^1\left(G_{(a,b,c)},F(k)\right) \to \ext^2(F,F) \to \ext^2((a+b+c+2)\cdot\op3,F(k)). $$
The group on the left vanishes because $H^1(F(t))=0$ for every $t\in\Z$, while the group on the right vanishes because $H^2(F(k))=0$. We conclude from \cite[Prop. 3.4]{H-r} that $\dim\ext^1(F,F)=8c_2(F)-3$, matching the dimension of the family as computed in the previous paragraph. It follows that the family of sheaves given by (\ref{newclass}) provides an component of the moduli space of stable rank 2 reflexive sheaves on $\p3$.

Summarizing the results obtained so far, we have the following theorem.

\begin{theorem}\label{thm newclass}
For each triple $(a,b,c)$ of positive integers such that $3a+2b+c$ is nonzero and even, the rank 2 reflexive sheaves given by (\ref{newclass}) fill out an irreducible, nonsingular, component $\cals(a,b,c)$ of $\calr(0;n;m)$ of expected dimension $8n-3$, where $n$ and $m$ are given by the expressions:
\begin{align*}
n = & \frac{1}{4}(3a+2b+c)^2 + \frac{3}{2}(3a+2b+c) - (b+c), \\
m=m(a,b,c) = & 27{{a+2}\choose{3}} + 8{{b+2}\choose{3}} + {{c+2}\choose{3}} + 3(3a+2b+5)ab + \\
    & + \frac{3}{2}(3a+c+4)ac + (2b+3c+3)bc + 6abc
\end{align*}
More precisely, let $\widetilde{\cals}(a,b,c)\subset \Hom\left(G_{(a,b,c)},(a+b+c+2)\cdot\op3\right)$ be the open subset consisting of monomorphisms with 0-dimensional degeneracy loci; then
$$ \cals(a,b,c) = \widetilde{\cals}(a,b,c) / (\Aut(G_{(a,b,c)})\times GL(a+b+c+2))/\CC^*) . $$
\end{theorem} 

Two particular cases deserve special attention, as they were previously considered by Chang in \cite{Chang}. First, we set $a=b=0$ and $c=2$, so that $n=2$ and $m=4$ and (\ref{newclass}) reducing to
\begin{equation}\label{n=2 m=4}
0 \to 2\cdot\op3(-1) \stackrel{\alpha}{\longrightarrow} 4\cdot\op3 \to F(1) \to 0 .
\end{equation}
It is shown in \cite[Lemma 2.9]{Chang} that every stable rank 2 reflexive sheaf $F$ with $c_2(F)=2$ and $c_3(F)=4$ admits a resolution of the form (\ref{n=2 m=4}); in other words, $\cals(0,0,2)=\calr(0;2;4)$.

The second case considered by Chang corresponds to $a=c=0$ and $b=1$, so that $n=3$ and $m=8$ and (\ref{newclass}) reducing to
\begin{equation}\label{n=3 m=8}
0 \to \op3(-2) \stackrel{\alpha}{\longrightarrow} 3\cdot\op3 \to F(1) \to 0 .
\end{equation}
One can check that every stable rank 2 reflexive sheaf $F$ with $c_2(F)=3$ and $c_3(F)=8$ admits a resolution of the form (\ref{n=3 m=8}), cf. \cite[proof of Theorem 3.9]{Chang}; in other words, $\cals(0,1,0)=\calr(0;3;8)$.

Finally, we are ready to establish the result promised in the beginning of the section.

\begin{theorem}\label{ein 0d}
Let $\zeta_n$ denote the number of irreducible components of $\calm(n)$ whose generic points correspond to sheaves with 0-dimensional singularities. Then \linebreak $\lim\sup_{n\to\infty}\zeta_n=\infty$.
\end{theorem}

\begin{proof}
For any integer $q\ge1$ set $n_q=9q^2-6q-1$ and for any integer $i,\ 0\le i\le q-1,$ set $a_{q,i}=i,\ b_{q,i}=3q-3i-3,\ c_{q,i}=3i+2$. Then, according to Theorem \ref{thm newclass}, the sheaf $F$ defined by (\ref{newclass}) for the triple of integers $(a,b,c)=(a_{q,i},b_{q,i},c_{q,i})$ belongs to an irreducible component $\cals_{q,i}=\cals(a_{q,i},b_{q,i},c_{q,i})$ of $\calr(0;n_q,m_{q,i})$, where $m_{q,i}=m(a_{q,i},b_{q,i},c_{q,i})$ is an even integer given by the second formula of Theorem \ref{thm newclass}. Now by Theorem \ref{0d comp's}, to each $\cals_{q,i}$ there corresponds an irreducible component  $\overline{\calt(n_q,\frac{m_{q,i}}{2})}$ of dimension $8n_q-3+2m_{q,i}$ in $\calm(n_q)$ whose generic point
is a sheaf with 0-dimensional singularities. Since
$0\le i\le q-1$, we therefore obtain $q$ distinct irreducible components of $\calm(n_q)$ with this property. In other words, in the notation of this theorem, $\zeta_{n_q}\ge q$. Hence, $\lim\sup_{n\to\infty}\zeta_n=\infty$.

\end{proof}

\section{Components of sheaves with 1-dimensional singularities} \label{1dsing}

Let $E$ be a rank 2 torsion free sheaf with 1-dimensional singularities, that is, the quotient sheaf $Q_E=E^{\vee\vee}/E$ has pure dimension one. Given any coherent sheaf $G$ on $\p3$, one has $\inext^3(E,G)=0$ and $\inext^2(E,G)\simeq \inext^3(Q_E,G)=0$ due to the reflexivity of $E^{\vee\vee}$. Therefore, torsion free sheaves $E$ with 1-dimensional singularities have homological dimension equal to 1; in other words, $E$ admits a locally free resolution of the form
\begin{equation}\label{E res dim 1}
0 \to L_1 \to L_0 \to E \to 0. 
\end{equation}

\begin{lemma}\label{chi ext tf lemma 2}
If $E$ is a rank 2 torsion free sheaf on $\p3$ with $c_1(E)=0$, and with 1-dimensional singularities,  then
$$ \chi(\inhom(E,E)) - \chi(\inext^1(E,E)) = \sum_{j=0}^3 (-1)^j\dim\ext^j(E,E) = - 8c_2(E) + 4. $$
\end{lemma}
\begin{proof}
In this case, the spectral sequence of local to global Ext's converges in the third page, and it yields
\begin{itemize}
\item[(i)] $\ext^1(E,E) = H^1(\inhom(E,E))\oplus\ker d^{01}_2$;
\item[(ii)] $\ext^2(E,E) = \coker d^{01}_2 \oplus \ker d^{11}_2$;
\item[(iii)] $\ext^3(E,E) = \coker d^{11}_2$
\end{itemize}
where $d^{01}_2$ and $d^{11}_2$ are the spectral sequence maps
\begin{equation} \label{d012-dim1}
d^{01}_2 : H^0(\inext^1(E,E)) \to H^2(\inhom(E,E)) ~~{\rm and}
\end{equation}
\begin{equation} \label{d112}
d^{11}_2 : H^1(\inext^1(E,E)) \to H^3(\inhom(E,E)) .
\end{equation}
The first equality is then an immediate consequence. 

As for the last equality, the same proof of \cite[Prop. 3.4]{H-r} applies here, since $E$ has homological dimension $1$.
\end{proof}

\begin{remark} \rm
We observe that the proof of the first equality does not depend on the hypotheses $\rk(E)=2$ and $c_1(E)=0$, being valid for any torsion free sheaf with 1-dimensional singularities.
\end{remark}

Recall that a rank 2 \emph{instanton sheaf} on $\p3$ is a rank 2 torsion free sheaf $E$ with $c_1(E)=0$ such that
$$ h^0(E(-1))=h^1(E(-2))=h^2(E(-2))=h^3(E(-3))=0. $$ 
These are precisely the sheaves obtained as cohomology of linear monads of the form (cf. \cite{J-i})
$$ 0 \to n\cdot \op3(-1) \stackrel{\alpha}{\longrightarrow} (2n+2)\cdot\op3 \stackrel{\beta}{\longrightarrow} n\cdot\op3(1) \to 0 . $$
The second Chern class of $E$ is called the \emph{charge} of $E$. An \emph{instanton bundle} is simply a locally free instanton sheaf. Let $\cali(n)$ denote the moduli space of instanton bundles of charge $n$; since every instanton bundle is $\mu$-stable, $\cali(n)$ can be regarded as an open subset of $\calm(n)$. Moreover, for each $n\ge1$, $\cali(n)$ is an irreducible \cite{T1,T2}, nonsingular \cite{JV}, affine \cite{CO} variety of dimension $8n-3$. The trivial sheaf $2\cdot\op3$ is considered the instanton bundle of charge 0; with this in mind, $\cali(0)$ consists of a single point. $\cali(n)$ is known to be rational for $n\le3$.

For $n>0$ denote by $\call(n)$ the union of those irreducible components of $\calm(n)$ whose generic points $F$ satisfy the condition 
$$ h^1(F(-2))=h^2(F(-2))=0. $$
We call $\call(n)$ the \textit{instanton stratum of} $\calm(n)$.

In this section we study the sheaves from $\calm(n)$ with 1-dimensional singularities obtained from the instanton bundles of charge $n-d$ by elementary transformations, in the sense of \cite[Section 3]{JMT}, along complete intersection curves of degree $d$ in $\p3$.

Let $d_1\le d_2$ be positive integers, and for $i=1,2$ let $S_{d_i}$ be a surface of degree $d_i$ in $\p3$. If the scheme $C_{d_1,d_2}:=S_{d_1}\cap S_{d_2}$ has pure dimension 1, we call it a \textit{complete intersection curve}. The degree, the arithmetic genus and the Hilbert polynomial of the curve $C=C_{d_1,d_2}$ are given by the formulas
\begin{equation}\label{d,pa,H}
\begin{split}
&d:=\deg C=d_1d_2,\ \ \ p_a(C)=1+\frac{d_1d_2(d_1+d_2-4)}{2},\\ 
&H(n)=\frac{d_1d_2(2n+4-d_1-d_2)}{2}.
\end{split}
\end{equation}
Let $\mathrm{Hilb}_{d_1,d_2}$ be an open subset of the Hilbert scheme $\mathrm{Hilb}_{H(t)}$ consisting of
reduced complete intersection curves $C_{d_1,d_2}$. This is a smooth irreducible scheme of dimension
\begin{equation}\label{dim H}
\begin{split}
\dim\mathrm{Hilb}_{d_1,d_2} & = 2\binom{d_1+3}{3}-4=h^0(N_{C/\p3}),\ \ \ if \ \ \ d_1=d_2, \\
\dim\mathrm{Hilb}_{d_1,d_2} & = \binom{d_1+3}{3}+\binom{d_2+3}{3}-\binom{d_2-d_1+3}{3}-2= \\
&=h^0(N_{C/\p3}),\ \ \ if\ \ \ d_1 < d_2,
\end{split}
\end{equation}
where $C\in\mathrm{Hilb}_{d_1,d_2}$. Besides, the $h^1$-cohomology of the sheaf $N_{C/\p3}$ is given by
\begin{equation}\label{hiN}
\begin{split}
&h^1(N_{C/\p3})=
2\binom{d_1-1}{3}+1,\ \ \ if\ \ \ d_1=d_2, \\
&h^1(N_{C/\p3})=
\binom{d_1-1}{3}+\binom{d_2-1}{3}-\binom{d_2-d_1-1}{3},
\ \ \ if\ \ \ d_1 < d_2,\\
&h^0(N_{C/\p3})-h^1(N_{C/\p3})=4d_1d_2.
\end{split}
\end{equation}

Let $\calh$ be an open dense subset of $\mathrm{Hilb}_{d_1,d_2}$ defined as
\begin{equation}\label{calh}
\begin{split}
&\calh=\calh_{d_1,d_2}:=
\{C\in\mathrm{Hilb}_{d_1,d_2}\ |\ C\ \textit{is a reduced curve with} \\
&\textit{at most ordinary singularities=simple double points}\}.
\end{split}
\end{equation}
Note that $\calh$ contains a dense open subset
\begin{equation}
\calh_s:=\{C\in\calh\ |\ C\ \textit{is a smooth irreducible curve}\}.
\end{equation}
Let 
\begin{equation}\label{univ family}
\mathcal{Z}\hookrightarrow\mathcal{H}\times\p3
\end{equation}
be the universal family of curves over $\mathcal{H}$. For any $C\in\calh$ denote $g:=p_a(C)$ and let 
\begin{equation} \label{hilb poly C}
P = P(n) = d_1d_2 \cdot n
\end{equation}
be the Hilbert polynomial of $C$ with respect to the sheaf $\calo_{\p3}(1)$. 

Consider the relative Jacobian functor $\mathbf{J}=\mathbf{J}^P:(Schemes/\calh)^o\to(Sets)$ defined as
$$
\mathbf{J}(T)=\{\textrm{invertible sheaves}\ F\ \textrm{on}\ \calz\times_{\calh}T\ \textrm{with fibrewise Hilbert polynomial}\ P\}/{\sim},
$$
where $F_1\sim F_2$ if there exists an invertible sheaf $N$ on $T$ such that $F_1\simeq F_2\otimes p^*N$,
for $p:\calz\times_{\calh}T\to T$ the projection.
Let $\mathbf{P}$ be the \'etale sheaf associated to $\mathbf{J}$. It is known (see \cite{A}, \cite[0.2]
{Esteves}) that $\mathbf{P}$ is represented by an algebraic space $\mathbb{P}$, locally of finite type over $\calh$. Furthermore, according to \cite[Theorem B]{Esteves} there exists an \'etale base change
\begin{equation}\label{sigma}
\sigma:\ \widetilde{\calh}\to\calh
\end{equation}
such that the functor $\widehat{\mathbf{J}}=\mathbf{J}\times_{\calh}\widetilde{\calh}$ is represented by a $\widetilde{\calh}$-scheme 
$$
\widehat{\mathbb{J}}\overset{\hat{\pi}}{\to}
\widetilde{\calh}
$$ 
together with the universal (Poincar\'e) line bundle
\begin{equation}\label{hat L}
\widehat{\mathbb{L}}\ \ \mathrm{on}\ \ 
\widehat{\mathbb{J}}\times_{\widetilde{\calh}}{\widetilde{\calz}},
\end{equation}
where $\widetilde{\calz}:=\calz\times_{\calh}{\widetilde{\calh}}$.
Consider an open subfunctor $\mathbf{J}^{ss}$ of $\mathbf{J}$ defined as
$$
\mathbf{J}^{ss}(T)=\{(F\textrm{mod}\sim)\in\mathbf{J}(T)
\ |\ F \ \textrm{is fibrewise}\ \calo_{\p3}(1)|_C\mbox{-semistable}\}
$$
The functor $\widetilde{\mathbf{J}}=\mathbf{J}^{ss}\times_{\calh}
\widetilde{\calh}$ is represented by a $\widetilde{\calh}$-scheme 
\begin{equation}
\widetilde{\mathbb{J}}\overset{\tilde{\pi}}{\to}
\widetilde{\calh}
\end{equation}
of finite type over $\widetilde{\calh}$, which is an open subscheme of $\widehat{\mathbb{J}}$ endowed with the universal (Poincar\'e) line bundle 
\begin{equation}\label{tilde L}
\widetilde{\mathbb{L}}=\widehat{\mathbb{L}}|_
{\widetilde{\mathbb{J}}\times_{\widetilde{\calh}}{\widetilde{\calz}}},
\end{equation}

On the other hand, $\mathbf{J}^{ss}$ is an open subfunctor of the moduli functor 
$\mathbf{M}=\mathbf{M}^P:(Schemes/\calh)^o\to(Sets)$,
$$
\mathbf{M}(T)=\{T\mbox{-flat sheaves}\ F\ \textrm{on}\
\ \calz\times_{\calh}T\ \textrm{with}\ \calo_{\p3}(1)|_C\mbox{-semistable}
$$
$$
\ \ \ \ \ \ \ \ \textrm{fibres over}\ T\ \textrm{having fibrewise Hilbert polynomial}\ P\}/{\sim},
$$
where by \cite{Simpson} (see also \cite[Section 4]{HL}) 
$\mathbf{M}$ is corepresented by a projective $\calh$-scheme
\begin{equation}
\mathbb{M}=\mathbb{M}^P_{\calz/\calh}
\overset{\pi}{\to}\calh;
\end{equation}\label{mathbb M}
respectively, $\mathbf{J}^{ss}$ is corepresented by a
quasi-projective $\calh$-scheme
\begin{equation}\label{mathbb M'}
\mathbb{M}'=\mathbb{M'}^P_{\calz/\calh}\overset{\pi'}
{\to}\calh
\end{equation}
which is an open subscheme of $\mathbb{M}$ and $\pi'=
\pi|_{\mathbb{M}'}$. 
Note that set-theoretically the schemes $\widetilde{\mathbb{J}}$ and $\mathbb{M}'$ are described as
\begin{equation}\label{set tilde J}
\begin{split}
&\widetilde{\mathbb{J}}=\{(C,w,[L])\ |\ C\in\calh,\ w\in\sigma^{-1}(C),\ L\ \textrm{is an invertible}\ 
 \\
&\calo_{\p3}(1)|_C\mbox{-semistable sheaf on}\ C\ 
\textrm{with Hilbert polynomial}\ P\},
\end{split}
\end{equation}
\begin{equation}\label{set M'}
\begin{split}
&\mathbb{M}'=\{(C,[L]_S)\ |\ C\in\calh,\ \ L\ \textrm{is an invertible}\ \calo_{\p3}(1)|_C\mbox{-semistable}\\
& \textrm{sheaf on}\ C\ \textrm{with Hilbert polynomial}\ P\},
\end{split}
\end{equation}
where $[L]_S$ denotes $S$-equivalence class of 
$L$ with respect to $\calo_{\p3}(1)|_C$.
Under this description, the corepresentability of
$\mathbf{J}^{ss}$ by $\mathbb{M}'$ implies that
there exists a surjective morphism of schemes 
\begin{equation}\label{varphi}
\varphi:\ \widetilde{\mathbb{J}}\to\mathbb{M}',\ \ \ (C,w,[L])\mapsto(C,[L]_S).
\end{equation}
Note that, since every invertible sheaf on a smooth (hence irreducible) curve $C\in\calh_s$ is $\calo_{\p3}(1)|_C$-stable, it follows that the functors $\mathbf{J}_s=\mathbf{J}^{ss}\times_{\calh}\calh_s$ and 
$\mathbf{M}\times_{\calh}\calh_s$ are isomorphic, and they are represented by the scheme
\begin{equation}\label{Ms}
\mathbb{M}_s=\mathbb{M}\times_{\calh}\calh_s
\overset{\pi_s}{\to}\calh_s,
\end{equation}
where $\pi_s=\pi|_{\mathbb{M}_s}$. Hence the functor 
$\widetilde{\mathbf{J}}_s=
\mathbf{J}_s\times_{\calh_s}\widetilde{\calh}_s$ is represented by the scheme 
\begin{equation}\label{tilde Js}
\widetilde{\mathbb{J}}_s:=\mathbb{M}_s\times_{\calh_s}
\widetilde{\calh}_s=\mathbb{M}\times_{\calh}
\widetilde{\calh}_s\overset{\tilde{\pi}_s}{\to}
\widetilde{\calh}_s
\end{equation}
of finite type over $\widetilde{\calh}_s$, 
which is an open subscheme of $\widetilde{\mathbb{J}}$. 
Note that, by construction, $\pi_s$ is a fibration
\begin{equation}\label{fibration pi s}
\pi_s:\ \mathbb{M}_s\to\calh_s,\ \ \ 
\pi^{-1}(C)=\mathrm{Pic}^{g-1}(C),\ \ \ C\in\calh_s,
\end{equation}
where $\mathrm{Pic}^{g-1}(C)=\{[L]\in\mathrm{Pic}(C)\ |\ \deg L=g-1\}$. This implies that $\mathbb{M}_s$ is smooth and  irreducible, since
$\calh_s$ is clearly smooth and irreducible. In addition,
\begin{equation}\label{fibration tilde pi s}
\tilde{\pi}_s:\ \widetilde{\mathbb{J}}_s\to
\widetilde{\calh}_s
\end{equation}
is also a fibration with fibre $\mathrm{Pic}^{g-1}(C)$ which is smooth since $\widetilde{\calh}_s$ is smooth as
an \'etale cover of $\calh_s$. 

Now consider the closure
\begin{equation}\label{M0}
\mathbb{M}^0:=\overline{\mathbb{M}}_s
\end{equation}
of the scheme $\mathbb{M}_s$ in $\mathbb{M}$. In the next section (see the proof of Lemma \ref{Lemma 5.1}(iv))
we will make use of the following lemma.
\begin{lemma}\label{Lemma 4.1}
$\mathbb{M}'\subset\mathbb{M}^0$. 
\end{lemma}
\begin{proof}
It is known (see, e.g., \cite[Section 0.2]{Esteves}, \cite[Fact 4.4]{MRV}) that the algebraic space $\mathbb{P}$ representing the functor $\mathbf{P}$ is formally smooth over $\calh$. This implies that the scheme $\widehat{\mathbb{J}}$, hence also the schemes $\widetilde{\mathbb{J}}$ and $\widetilde{\mathbb{J}}_s$,
are formally smooth over $\widetilde{\calh}$.

Take a point $x=(C,[L]_S)\in\mathbb{M}'$. By definition, $L$ is an invertible sheaf on $C$. We have to show that $x\in\mathbb{M}^0$. For this, let $\tilde{x}\in\widetilde{\mathbb{J}}$ be any point in the fibre $\varphi^{-1}(x)$ where $\varphi$ is defined in (\ref{varphi}) and let $w=\tilde{\pi}(\tilde{x})$.
Refining the \'etale base change $\sigma$, we may assume $\widetilde{\calh}=\sqcup\widetilde{U}_i$, where each $\widetilde{U}_i=\sigma^{-1}(U_i)$ is irreducible and $\cup U_i$ is an open cover of $\calh$. The point $w\in\widetilde{\calh}$ lies in some $\widetilde{U}_i$, and let $X$ be any irreducible component of $\tilde{\pi}^{-1}(U_i)$ containing $\tilde{x}$.
Since $\calh_s$ is an irreducible dense open subset of $\calh$, it follows that $U_{is}=U_i\cap\calh_s$ is a
dense open subset of $U_i$. Hence, $\widetilde{U}_{is}=\sigma^{-1}(U_{is})$ is a dense open subset.

Next, as $\tilde{\pi}:\widetilde{\mathbb{J}}\to\widetilde{\calh}$
is formally smooth, $\tilde{\pi}|_X:X\to\widetilde{U}_i$ is dominant.
Hence, $X'=X\cap\tilde{\pi}^{-1}(\widetilde{U}_{is})$ is dense open both in $X$ and $\tilde{\pi}^{-1}(\widetilde{U}_{is})$. Thus, $\varphi(X')$ is dense in $\varphi(X)$ and, by
construction (see (\ref{tilde Js})-(\ref{fibration tilde pi s})), $\varphi(X')$ lies in $\mathbb{M}_s$ 
and contains a nonempty open subset of $\mathbb{M}_s$. Since $\mathbb{M}_s$ is irreducible, 
$\mathbb{M}^0=\overline{\mathbb{M}}_s=\overline{\varphi(X')}=\overline{\varphi(X)}$, and, by
construction, $x\in\overline{\varphi(X)}$.
\end{proof}

Note that, since $\widetilde{\mathbb{J}}$ is formally smooth over $\widetilde{\calh}$, it follows that
$\widetilde{\mathbb{J}}_s$ is dense and open in $\widetilde{\mathbb{J}}$, hence (\ref{fibration tilde pi s}) implies that
\begin{equation}\label{dim J}
\dim\widetilde{\mathbb{J}}= \dim\widetilde{\mathbb{J}}_s=
1+\frac{d_1d_2(d_1+d_2-4)}{2}+\dim\calh,
\end{equation}
where $\dim\calh$ is given by (\ref{dim H}).

Take any curve $C\in\calh_s$. Then the set $U_C:=\{[L]\in\mathrm{Pic}^{g-1}(C)\ |\ h^0(L)=h^1(L)=0\}$ is dense and open in $\mathrm{Pic}^{g-1}(C)$ since it is the complement of the divisor $\Theta=\im(a)$, where 
$a:S^{g-1}C\to\mathrm{Pic}^{g-1}(C),\ D\mapsto\calo_C(D)$ is the Abel--Jacobi map. Therefore, denoting
$$
\mathbb{J}:=\{(C,[L])\in\mathbb{M}_s\ |\ h^0(L)=h^1(L)=0\},
$$
$$
\widetilde{\mathbb{J}}_0:=\{(C,w,[L])\in \widetilde{\mathbb{J}}_s\ |\ h^0(L)=h^1(L)=0\}=
\mathbb{J}\times_{\calh_s}\widetilde{\calh}_s,
$$
we obtain that $\mathbb{J}$ is a nonempty and, by semicontinuity, open subset of $\mathbb{M}_s$, which is dense and smooth as $\mathbb{M}_s$ is smooth and irreducible. Similarly, $\widetilde{\mathbb{J}}_0$ is smooth, dense and open in $\widetilde{\mathbb{J}}_s$. Note also that by (\ref{tilde L}),
${\widetilde{\mathbb{J}}_0\times_{\widetilde{\calh}}{\widetilde{\calz}}}$
carries a universal (Poincar\'e) line bundle, which is
\begin{equation}\label{tilde L0}
\widetilde{\mathbb{L}}_0=\widetilde{\mathbb{L}}|_
{\widetilde{\mathbb{J}}_0\times_{\widetilde{\calh}}{\widetilde{\calz}}}.
\end{equation}

Next, for $c\ge0$ and any point $([F],C,[L])\in\cali(c)\times\mathbb{J}$,
set 
$$
\mathbb{P}\mathrm{Hom}(F,C,L)_e:=\{\mathbf{k}\varphi\in \mathbb{P}(\mathrm{Hom}(F,L(2)))\ |\ \varphi:F\to L(2) 
\ \textrm{is an epimorphism}\}. 
$$
Recall that $\cali(0)=\{\mathrm{pt}\}$. 

\begin{lemma}\label{Lemma 4.2}
For each $c\ge0$, there is a smooth, dense, and open subset $(\cali(c)\times\mathbb{J})_e^0$
of $\cali(c)\times\mathbb{J}$ such that, for any $([F],C,[L])\in(\cali(c)\times\mathbb{J})_e^0$, 
one has: 
\begin{itemize}
\item[(i)] $h^i(L)=h^i(L^{-1}\otimes\omega_C)=0,\ i=0,1;$
\item[(ii)] $h^1(F\otimes L(2)) = h^1(F\otimes(L^{-1}\otimes\omega_C)(2))=0$;
\item[(iii)] $\mathbb{P}\mathrm{Hom}(F,C,L)_e$ is a dense open subset of
$\mathbb{P}(\mathrm{Hom}(F,L(2)))$;
\item[(iv)]
\begin{equation}\label{dim Phom}
\dim\mathbb{P}\mathrm{Hom}(F,C,L)_e=4d_1d_2-1.
\end{equation}
\item[(v)] there is a smooth, dense, and open subset $(\cali(c)\times\mathbb{J})_e$
of $\cali(c)\times\mathbb{J}$ containing $(\cali(c)\times\mathbb{J})_e^0$ and such that, for any
$([F],C,[L])\in(\cali(c)\times\mathbb{J})_e$, the statements (iii) and (iv) and the equalities
$h^0(L)=h^1(L)=h^1(F\otimes L(2))=0$ from (i) and (ii) above are true.
\end{itemize}
\end{lemma}

\begin{proof}
Take a point $(C,[L])\in\mathbb{J}_0\cap\mathbb{J}_s$, so that
\begin{equation}\label{h1}
h^i(L)=h^i(L^{-1}\otimes\omega_C)=0,\ i=0,1.
\end{equation}

We first consider the case $c=0$, so that $F\otimes L(2) \simeq 2\cdot L(2)$ and 
$\mathrm{Hom}(F,L(2)) \simeq H^0(2\cdot L(2))$. Items (ii), (iii) and (v) follow immediately. As for item (iv), just note that $\chi(L(k))=d_1d_2\cdot k$ (since $\chi(L)=0$ and $\deg C=d_1d_2$), thus $h^0(L(2))=\chi(L(2))=2d_1d_2$. 

Next, let $c>0$; take a ${}^{,}$t Hooft bundle $[F]\in\cali(c)$, i. e. a bundle fitting in an exact triple
$$ 0\to\calo_{\p3}(-1)\to F\to\cali_Y(1)\to0~,$$
where $Y$ is a union of $c+1$ disjoint lines in $\p3$. Choose $Y$ in such a way that $Y\cap C=\emptyset$. Then tensoring the above triple with $L(2)$, we obtain exact triples 
\begin{equation}\label{L's}
\begin{split}
& 0\to L(1)\to F \otimes L(2)\to L(3)\to0,\\
& 0\to(L^{-1}\otimes\omega_C)(1)\to F\otimes(L^{-1}\otimes\omega_C)(2)\to(L^{-1}\otimes\omega_C)(3)\to0.
\end{split}
\end{equation}
The equalities (\ref{h1}) imply
\begin{equation}\label{h1,h1}
h^1(L(1))=h^1(L(3))=0,\ \ \ 
h^1((L^{-1}\otimes\omega_C)(1))=h^1((L^{-1}\otimes
\omega_C)(3))=0,
\end{equation}
so that (\ref{L's}) yields
$$
h^1(F\otimes L(2))=0,\ \ \ 
h^1(F\otimes(L^{-1}\otimes\omega_C)(2))=0
$$ 
for the above chosen point
$([F],C,[L])\in\cali(c)\times\mathbb{J}$. Since, by
semicontinuity, the vanishing of $h^1(F\otimes L(2))$ and 
$h^1(F\otimes(L^{-1}\otimes\omega_C)(2))$ 
is an open condition on 
$([F],C,[L])\in \cali(c)\times\mathbb{J}$, 
it follows that the set
\begin{equation}\label{IJ'}
\begin{split}
& (\cali(c)\times\mathbb{J})'=
\{([F],C,[L])\in\cali(c)\times\mathbb{J}\ |\ h^i(L)=h^i(L^{-1}\otimes\omega_C)=0,\ \\ 
& \ h^1(F\otimes L(2))=h^1(F\otimes(L^{-1}
\otimes\omega_C)(2))=0,\ i=0,1,\}
\end{split}
\end{equation}
is a nonempty (hence dense) open subset of $\cali(c)\times\mathbb{J}$.
Next, from (\ref{h1}) and (\ref{L's}) we obtain the exact
sequence  
\begin{equation}\label{h0(L's)}
0\to H^0(L(1))\to H^0(F\otimes L(2))\overset{\varepsilon}{\to} H^0(L(3))\to0.
\end{equation}
Since the sheaves $\calo_{\p3}(1)$ and $\calo_{\p3}(3)$ are very ample, it follows from (\ref{h1}) that the linear series $|L(1)|$ and $|L(3)|$ on $C$ have no fixed
points. This implies that there exist such sections
$s_i\in H^0(L(i)),\ i=1,3,$ that
$$
(s_1)_0\cap (s_3)_0=\emptyset.
$$
Take any section $s'\in\varepsilon^{-1}(s_3)$, where $\varepsilon$ is the epimorphism in (\ref{h0(L's)}). Then the last equality implies that the section $s:=s'+s_1\in H^0(F\otimes L(2))$ has no zeroes. Hence its transpose
$\varphi={}^{\sharp}s:F\simeq F^{\vee}\to L(2)$
is an epimorphism, i. e. 
\begin{equation}\label{not empty}
\mathbb{P}\mathrm{Hom}(F,C,L)_e\ne\emptyset.
\end{equation}
Since $\mathbb{P}\mathrm{Hom}(F,C,L)_e$ is an open subset of
the irreducible space $\mathbb{P}(\mathrm{Hom}(F,L(2)))$, it is dense in $\mathbb{P}(\mathrm{Hom}(F,L(2)))$.
Moreover, (\ref{not empty}) is an open condition on the
point $([F],C,[L])$ in $(\cali(c)\times\mathbb{J})'$.
Thus in view of (\ref{IJ'}) there exists a dense open subset
$(\cali(c)\times\mathbb{J})_e$ of $(\cali(c)\times\mathbb{J})'$ (hence of
$\cali(c)\times\mathbb{J}$) for which the statements (i)
- (iii) of Lemma hold. Besides, the smoothness of 
$(\cali(c)\times\mathbb{J})_e$ follows from that of $\cali(c)$ (see \cite{JV}) and of
$\mathbb{J}$.

Next, since $F\simeq F^{\vee}$, we have
$$ \dim\mathbb{P}\mathrm{Hom}(F,C,L)_e= \dim(\mathrm{Hom}(F,L(2)))-1=h^0(F\otimes L(2))-1 $$
Note that 
$$ h^0(F\otimes L(2))=h^0(L(1))+h^0(L(3))=\chi(L(1))+\chi(L(3))=4d_1d_2 ~,$$
where the first equality follows from the exact sequence (\ref{h1,h1}), while the second equality follows from
(\ref{h0(L's)}). Putting the last two equations together, we obtain (\ref{dim Phom}). 

At last, the statement (v) is clear by semicontinuity.
\end{proof}

In particular, note that $(\cali(0)\times\mathbb{J})_e=\mathbb{J}$.

Next, using Lemma \ref{Lemma 4.2} consider, for each $c\ge1$ and $d_2\ge d_1\ge1$, the set
\begin{equation}\label{W}
\widetilde{W}(d_1,d_2,c):=\{([F],C,[L],\mathbf{k}\varphi)\ |\ ([F],C,[L])\in(\cali(c)\times\mathbb{J})_e,\ 
\mathbf{k}\varphi\in\mathbb{P}\mathrm{Hom}(F,C,L)_e\}
\end{equation}
and the surjective projection
\begin{equation}\label{pi}
\pi:\ \widetilde{W}(d_1,d_2,c)\to
(\cali(c)\times\mathbb{J})_e,\ \
([F],C,[L],\mathbf{k}\varphi)\mapsto([F],C,[L])
\end{equation}
with fibre
\begin{equation}\label{fibre of pi}
\pi^{-1}([F],C,[L])=\mathbb{P}\mathrm{Hom}(F,C,L)_e
\overset{\textrm{open}}{\hookrightarrow}\mathbb{P}(\mathrm{Hom}
(F,L(2))).
\end{equation}

When $c=0$, one must also quotient out by the action of $GL(2)$ on the trivial sheaf $2\cdot\op3(2)$ in order to obtain a family of isomorphism classes of torsion free sheaves. Therefore, we define:
\begin{equation}\label{W(0)}
\widetilde{W}(d_1,d_2,0):=\{(C,[L],\mathbf{k}\varphi)\ |\ (C,[L])\in\mathbb{J}^0,\ 
\mathbf{k}\varphi\in\mathbb{P}\Hom(2\cdot\op3,C,L)_e/\mathbb{P}GL(2)\} .
\end{equation}

Also denote
\begin{equation}\label{W0}
\widetilde{W}(d_1,d_2,c)^0:=\pi^{-1}((\cali(c)\times\mathbb{J})_e^0).
\end{equation}

\begin{remark}\label{Remark 4.2} \rm
Note that $\widetilde{W}(d_1,d_2,c)$ with $c\ge1$ is a dense open subset of a Severi--Brauer variety fibered over $\mathcal{B}:=(\cali(c)\times\mathbb{J})_e$ with fibers, given by ${\mathbb{P}}^{4d_1d_2-1}$ via the projection $\pi$.

Indeed, let $\widetilde{\mathbf{V}}:=\mathcal{B}\times_{\mathcal{H}}\widetilde{\mathcal{H}}\to\mathcal{B}$ be an \'etale covering induced by (\ref{sigma}). According to \cite[Section 4]{HL}, over $\cali(c)$ there exists (locally in the \'etale topology) a universal rank-2 vector bundle. This means that there exists an open \'etale covering $\Phi:W\to \cali(c)$ and a rank 2 vector bundle $\mathbf{E}$ over $\p3\times W$ such that, for any $w\in W$, $\mathbf{E}|_{\p3\times w}\simeq E_t$, where $t=\Phi(w)\in\cali(c)$ and $E_t$ denotes the instanton bundle whose isomorphism class is represented by $t$. Let $\widetilde{\mathbf{U}}:=W\times_{\cali(c)}
\widetilde{\mathbf{V}}$, and let $\widetilde{\mathbf{\Gamma}}:=\widetilde{\mathbf{U}}\times_{\widetilde{\mathcal{H}}}\widetilde{\mathcal{Z}}$; let 
$\boldsymbol{\iota}:\widetilde{\mathbf{\Gamma}}\hookrightarrow\widetilde{\mathbf{U}}\times\p3$
be the lift into $\widetilde{\mathbf{U}}\times\p3$ of the universal family of curves $\mathcal{Z}$. Let also
$\mathbf{E}_{\widetilde{\mathbf{U}}}$ be the lift into $\widetilde{\mathbf{U}}\times\p3$
of the sheaf $\mathbf{E}$ and let $\mathbf{L}$ be the lift onto $\widetilde{\mathbf{\Gamma}}$
of the sheaf $\widetilde{\mathbb{L}}_0$ defined in (\ref{tilde L0}). We thus obtain a vector bundle 
$\boldsymbol{\tau}:= \mathcal{H}om_{\widetilde{\mathbf{\Gamma}}/\widetilde{\mathbf{U}}}
(\mathbf{E}_{\widetilde{\mathbf{U}}},\boldsymbol{\iota}_*\mathbf{L}(2))$ over
$\widetilde{\mathbf{U}}$, the fiber of which over a point $u\in\widetilde{\mathbf{U}}$ lying over a point
$([F],C,[L])\in\mathcal{B}$ is by construction isomorphic to $\mathrm{Hom}(F,L(2))$. Hence by (\ref{dim Phom}) the associated projective bundle $\mathbf{P}\boldsymbol{\tau}\to\widetilde{\mathbf{U}}$ 
is a ${\mathbb{P}}^{4d_1d_2-1}$-fibration. Applying to it the argument from the proof of Proposition 6.4 from \cite{JMT} we obtain that this fibration descends to a Severi--Brauer variety 
$\mathbf{p}_{\mathcal{B}}:\mathbf{P}_{\mathcal{B}}\to \mathcal{B}$ with fibers ${\mathbb{P}}^{4d_1d_2-1}$
over $\mathcal{B}$ such that, by the above, for any point $([F],C,[L])\in\mathcal{B}$ one has
\begin{equation}\label{fibre of pB}
\mathbf{p}_{\mathcal{B}}^{-1}([F],C,[L])=
\mathbb{P}(\mathrm{Hom}(F,L(2))).
\end{equation}
This, together with (\ref{fibre of pi}), shows that the variety $\mathbf{P}_{\mathcal{B}}$ contains  $\widetilde{W}(d_1,d_2,c)$ as a dense open subset. 

Finally, for the case $c=0$, note that although the fibers of the projection $\pi:\widetilde{W}(d_1,d_2,0)\to\mathbb{J}^0$ are not open subsets of a projective space, they are still smooth.
\end{remark}

From the previous remark and the smoothness of $(\cali(c)\times\mathbb{J})_e$ (see Lemma 
\ref{Lemma 4.2}) we obtain the following statement.

\begin{theorem}\label{Thm 4.3}
For each $c\ge0$ and $d_2\ge d_1\ge 1$, $\widetilde{W}(d_1,d_2,c)$ has a natural structure of a smooth integral scheme of dimension
\begin{equation}\label{dim W}
\dim\widetilde{W}(d_1,d_2,c) = 8c-3+\frac{1}{2}d_1d_2(d_1+d_2+4)+\dim\calh 
\end{equation}
where $\dim\calh$ is given by (\ref{dim H}), and, for $c\ge1$, the map 
$$ \pi:\ \widetilde{W}(d_1,d_2,c)\to(\cali(c) \times\mathbb{J})_e $$
defined in (\ref{pi}) is a morphism. Respectively, $\widetilde{W}(d_1,d_2,c)^0$ is a dense open subscheme of $\widetilde{W}(d_1,d_2,c)$.
\end{theorem}
\begin{proof}
It is enough to prove (\ref{dim W}). For $c\ge1$, since $\dim\cali(c)=8c-3$, (\ref{dim W}) follows from (\ref{dim J}) and (\ref{dim Phom}). For $c=0$, one easily sees from (\ref{W(0)}) and (\ref{dim J}) that 
$$ \dim \widetilde{W}(d_1,d_2,0) = \dim \mathbb{J} + 4d_1d_2-4 = \frac{1}{2}d_1d_2(d_1+d_2+4)+\dim\calh - 3, $$
as desired.
\end{proof}

Now for any point $\mathbf{w}=([F],C,[L],\mathbf{k}\varphi)\in \widetilde{W}(d_1,d_2,c)$ set
$$
E(\mathbf{w}):=\ker(F\overset{\varphi}
{\twoheadrightarrow}L(2)).
$$
By definition, we have an exact triple
\begin{equation}\label{standard tr}
0\to E(\mathbf{w})\to F\overset{\varphi}{\to}L(2)\to0.
\end{equation}

One easily checks, using the irreducibility of $C$, that $E(\mathbf{w})$ is a stable sheaf (see \cite[Corollary 4.2 and Lemma 4.3]{JMT}) and, in fact, $[E(\mathbf{w})]\in\calm(c+d_1d_2)$. Moreover, Lemma \ref{Lemma 4.2}(i) and the triple (\ref{standard tr}) twisted by $\calo_{\p3}(-2)$ yield
\begin{equation}
[E(\mathbf{w})]\in\call(c+d_1d_2).
\end{equation}

Given a point $(C,[L])\in\mathbb{J}$, we call the invertible $\calo_C$-sheaf $L$ a \textit{theta -characteristic} on $C$ if
$$ L^{\otimes2}\simeq\omega_C. $$ 

Consider a subset of $\mathbb{J}$ defined as
\begin{equation}
\Theta_{\mathbb{J}}:=\{(C,[L])\in\mathbb{J}\ |\ L\ \textrm{is a theta-characteristic\ on}\ C\}.
\end{equation}
It is a locally closed subset of $\mathbb{J}$. (Indeed,
$\Theta_{\mathbb{J}}$ is a fixed point set of an involution $\mathbb{J}\to\mathbb{J},\ (C,[L])\mapsto
(C,[\omega_C\otimes L^{-1}])$.)

Denote
$$
\Theta_W(d_1,d_2,c):=\pi^{-1}((\cali(c)\times\mathbb{J})_e\cap(\cali(c)\times\overline{\Theta}_{\mathbb{J}})),
$$
$$
W(d_1,d_2,c):=\widetilde{W}(d_1,d_2,c)^0\setminus
\Theta_W(d_1,d_2,c),
$$
where $\overline{\Theta}_{\mathbb{J}}$ is the closure of
$\Theta_{\mathbb{J}}$ in $\mathbb{J}$.
By definition, $W(d_1,d_2,c)$ an open subset of $\widetilde{W}(d_1,d_2,c)$. Since for $p_a(C)>0$ the set $\overline{\Theta}_{\mathbb{J}}$ is clearly a proper closed subset of $\mathbb{J}$, it follows that for $p_a(C)>0$ the set $W(d_1,d_2,c)$ is a dense open subset of $\widetilde{W}(d_1,d_2,c)$.

\begin{proposition}\label{prop 4.5}
For an arbitrary closed point $\mathbf{w}=([F],C,[L],\mathbf{k}\varphi)\in W(d_1,d_2,c)$
with $c\ge0$, and $(d_1,d_2)\ne(1,1)$, $(d_1,d_2)\ne(1,2)$, the sheaf $E=E(\mathbf{w})$ satisfies the relations:
\begin{equation}\label{dim Ext2}
\dim\mathrm{Ext}^2(E,E)=h^1(N_{C/\p3})+p_a(C)-1,
\end{equation}
\begin{equation}\label{dim Ext1}
\dim\mathrm{Ext}^1(E,E)=h^1(N_{C/\p3})+ p_a(C)-1+8(c+d_1d_2)-3,
\end{equation}
where $p_a(C)$ and $h^1(N_{C/\p3})$ are given by (\ref{d,pa,H}) and (\ref{hiN}), respectively.
\end{proposition}
\begin{proof}
The conditions $(d_1,d_2)\ne(1,1)$, $(d_1,d_2)\ne(1,2)$ imply that $p_a(C)>0$, so that $W(d_1,d_2,c)$ is non-empty. Apply the functor $\mathrm{Hom}(L(2),-)$ to the triple
(\ref{standard tr}):
\begin{equation}\label{long1}
\begin{split}
&...\to\mathrm{Ext}^1(L(2),L(2))\overset{\delta}{\to}
\mathrm{Ext}^2(L(2),E)\\
&\to\mathrm{Ext}^2(L(2),F)\to
\mathrm{Ext}^2(L(2),L(2))\to\mathrm{Ext}^3(L(2),E)\\
&\to\mathrm{Ext}^3(L(2),F)\to...
\end{split}
\end{equation}
Next, apply the functor $\mathcal{H}om(L(2),-)$ to (\ref{standard tr}). Using the vanishing of the
sheaves $\inhom(L(2),F)$,  $\inext^1(L(2),F)$ and 
$\inext^i(F,L(-2)),\ i=1,2,$
(note that $\dim L(2)=1$
and $F$ is locally free on $\p3$) we obtain an isomorphism $\partial_1:\mathcal{H}om(L(2),L(2))\overset{\simeq}{\to}
\inext^1(L(2),E)$ 
and an exact sequence
\begin{equation}\label{d2}
0\to\inext^1(L(2),L(2))\overset{\partial_2}{\to}
\inext^2(L(2),E)\to\inext^2(L(2),F)
\overset{\varepsilon}{\to}\inext^2(L(2),L(2)) 
\end{equation}

Respectively, applying the functor $\inhom(\cdot,L(-2))$ to the triple (\ref{standard tr}) yields an isomorphism
\begin{equation}\label{psi1}
\psi:\ \inext^1(E,L(-2))\simeq
\inext^1(L(2),L(-2)).
\end{equation}

The isomorphisms
$h^1(\partial):\ H^1(\mathcal{H}om(L(2),L(2)))\overset{\simeq}{\to}H^1(\inext^1(L(2),E))$, 
the homomorphism $\delta$ in (\ref{long1}) and the monomorphism 
$h^0(\partial_2):\ H^0(\inext^1(L(2),L(2)))\to H^0(\inext^2(L(2),E))$ induced by (\ref{d2})
fit in the commutative diagram
\begin{equation}\label{diag1}
\xymatrix{
0\ar[d] & 0\ar[d] \\
H^1(\mathcal{H}om(L(2),L(2)))
\ar[r]^-{h^1(\partial_1)}_-{\simeq} \ar[d] &
H^1(\inext^1(L(2),E))\ar[d]^j\\
\mathrm{Ext}^1(L(2),L(2))
\ar[r]^-{\delta}\ar[d] & 
\mathrm{Ext}^2(L(2),E)\ar[d] \\
H^0(\inext^1(L(2),L(2)))\ar[r]^-{h^0(\partial_2)}\ar[d] &
H^0(\inext^2(L(2),E))\ar[d] \\
0 &\ 0,}
\end{equation}
in which the vertical exact triples come from the spectral sequences 
$$ H^p(\inext^q(L(2),L(2)))\Rightarrow \mathrm{Ext}^{\bullet}(L(2),L(2)) ~{\rm and}~ $$
$$ H^p(\inext^q(L(2),E))\Rightarrow \mathrm{Ext}^{\bullet}((L(2),E), $$
respectively.
Now restrict the triple (\ref{standard tr}) onto the curve $C$. Using the relation 
$\det F\otimes\calo_C \simeq\calo_C$ we obtain an exact triple
$0\to L^{-1}(-2)\to F\otimes\calo_C\to L(2)\to0$.
Tensoring this triple with the invertible $\calo_C$-sheaf
$\inext^2(L(2),\calo_C)$ and using the isomorphisms
$\inext^2(L(2),\calo_C)\otimes L^{-1}(-2)\simeq \inext^2(L(2),L^{-1}(-2))$,
$\inext^2(L(2),\calo_C)\otimes F\simeq
\inext^2(L(2),F)$,
$\inext^2(L(2),\calo_C)\otimes L(2)\simeq
\inext^2(L(2),L(2))$ we obtain the exact triple
$$
0\to\inext^2(L(2),L^{-1}(-2))\to
\inext^2(L(2),F)\overset{\varepsilon}{\to}\inext^2(L(2),L(2))\to0.
$$
This triple together with (\ref{d2}) yields an exact triple
\begin{equation}\label{triple0}
0\to\inext^1(L(2),L(2))\overset{\partial_2}{\to}
\inext^2(L(2),E)\to\inext^2(L(2),L^{-1}(-2))
\to0.
\end{equation}
Note that, since $L$ is not a theta-characteristic on $C$,
it follows that the sheaf
\begin{equation}\label{nontrivial}
\begin{split}
& \inext^2(L(2),L^{-1}(-2))\simeq
\inext^2(\calo_C,\calo_C)\otimes L^{-2}(-4)\simeq \\
& \det N_{C/\p3}\otimes\omega_{\p3}\otimes L^{-2}\simeq\omega_C\otimes L^{-2}
\end{split}
\end{equation}
is an invertible $\calo_C$-sheaf of degree 0, nonisomorphic to $\calo_C$,
hence it has no sections. Thus the above triple
gives an isomorphism
\begin{equation}\label{isom'm1}
H^0(\inext^1(L(2),L(2)))\overset{h^0(\partial_2)}{\simeq}H^0(\inext^2(L(2),E)).
\end{equation} 
This is the lower horizontal isomorphism in the diagram (\ref{diag1}) from which it follows that the homomorphism $\delta$ is an
isomorphism:
\begin{equation}\label{isom delta}
\delta:\ \mathrm{Ext}^1(L(2),L(2))\overset{\simeq}{\to}
\mathrm{Ext}^2(L(2),E).
\end{equation}

Next, from Lemma \ref{Lemma 4.2}(i) and the triple (\ref{L's}) twisted by $\calo_{\p3}(-4)$ it follows easily
that $H^0(F\otimes L(-2))=0$, and the Serre--Grothendieck duality together with the isomorphism
$F\simeq F^{\vee}$ implies
\begin{equation}\label{Ext3=0}
\mathrm{Ext}^3(L(2),F)\simeq
\mathrm{Hom}(F,L(-2))^{\vee}=0.
\end{equation}
Similarly, since $\dim C=1$, it follows that
$\mathrm{Ext}^1(L(2),F(-4))\simeq
H^2(F\otimes L(-2))=0$. Thus the exact sequence
$$ 0\to\mathrm{Hom}(F,F(-4))\to \mathrm{Hom}(E,F(-4))\to \mathrm{Ext}^1(L(2),F(-4)) $$
together with the equality $\mathrm{Hom}(F,F(-4))=0$ (note that $F$ is
stable) yields $\mathrm{Hom}(E,F(-4))=0$, and again
by Serre--Grothendieck duality we obtain 
\begin{equation}\label{another Ext3=0}
\mathrm{Ext}^3(F,E)=0.
\end{equation}

Next, twisting the triple (\ref{standard tr}) with $F^{\vee}\simeq F$ and passing to cohomology
we obtain an exact sequence 
$H^1(F\otimes L(2))\to H^2(F^{\vee}\otimes E)\to H^2(F^{\vee}\otimes F)$. Using the vanishing of
$H^2(F^{\vee}\otimes F)$ (see \cite{JV}) and of $H^1(F\otimes L(2))$ (Lemma \ref{Lemma 4.2}(ii)) we get since $F$ is locally free:
\begin{equation}\label{Ext2=0}
\mathrm{Ext}^2(F,E)\simeq H^2(F^{\vee}\otimes E)=0.
\end{equation}
Now apply the functor $\mathrm{Hom}(-,E)$ to the triple (\ref{standard tr}) and use (\ref{another Ext3=0}) and (\ref{Ext2=0}) to obtain the isomorphism
\begin{equation}\label{Ext2(E,E)}
\mathrm{Ext}^2(E,E)\simeq\mathrm{Ext}^3(L(2),E).
\end{equation}
The sequence (\ref{long1}) together with (\ref{isom delta}), (\ref{Ext3=0}) and (\ref{Ext2(E,E)}) yields an exact sequence
\begin{equation}
0\to\mathrm{Ext}^2(L(2),F)\to\mathrm{Ext}^2(L(2),L(2))
\to\mathrm{Ext}^2(E,E)\to0.
\end{equation}

Next, since $F\simeq F^{\vee}$ is locally free, the
Serre--Grothendieck duality on $\p3$ and on $C$ yields:
\begin{equation}\label{Ext2(l2,cale)}
\mathrm{Ext}^2(L(2),F)\simeq H^1(F^{\vee}\otimes
L(-2))^{\vee}\simeq H^0(F\otimes(L^{-1}
\otimes\omega_C)(2))^{\vee}.
\end{equation}
On the other hand, Riemann-Roch for the sheaf 
$F\otimes(L^{-1}\otimes\omega_C)(2)$ and Lemma 
\ref{Lemma 4.2}(ii) imply
$h^0(F\otimes(L^{-1}\otimes\omega_C)(2))^{\vee}=
4d_1d_2$, hence (\ref{Ext2(E,E)}) and (\ref{Ext2(l2,cale)}) yield
\begin{equation}\label{dimExt2EE}
\dim\mathrm{Ext}^2(E,E)=\dim\mathrm{Ext}^2(L,L)-4d_1d_2.
\end{equation}

Next, using the fact  that $H^2(\inext^2(L,L))=0$ since $\dim\inext^2(L,L)=1$, we obtain that the
spectral sequence $H^p(\inext^q(L,L))\Rightarrow
\mathrm{Ext}^{\bullet}(L,L)$ yields an exact triple 
$$
0\to H^1(N_{C/\p3})\to\mathrm{Ext}^2(L,L)\to H^0(\inext^2(L,L))\to0. 
$$         
Note that in view the Serre duality $h^1(\calo_C(d_1+d_2))=h^0(\calo_C(-4))=0$, 
and the isomorphisms $\inext^2(L,L)\simeq\inext^2\calo_C,\calo_C)\simeq\det N_{C/\p3}\simeq\calo_C(d_1+d_2)$ we obtain by
Riemann-Roch 
$h^0(\inext^2(L,L))=h^0(\calo_C(d_1+d_2))=
\chi(\calo_C(d_1+d_2))=d_1d_2(d_1+d_2)+1-p_a(C)$.
This together with the above triple yields
$$
\dim\mathrm{Ext}^2(L,L)=h^1(N_{C/\p3})+d_1d_2(d_1+d_2)+
1-p_a(C).
$$
Now (\ref{dim Ext2}) follows by substituting the last equality in (\ref{dimExt2EE}) and using (\ref{d,pa,H}).
The equality (\ref{dim Ext1}) follows from here in view of the relation $\dim\ext^1(E,E)-\dim\ext^2(E,E)=8c_2(E)-3=8(c+d_1d_2)-3$ (see Lemma \ref{chi ext tf lemma 2} above).
\end{proof}

Consider the map 
\begin{equation}\label{f}
f:\ \widetilde{W}(d_1,d_2,c)\to\call(c+d_1d_2),\ \mathbf{w}\mapsto[E(\mathbf{w})].
\end{equation}
Using Remark \ref{Remark 4.2} and Theorem \ref{Thm 4.3}
one easily sees that $f$ is a morphism.

\begin{theorem}\label{Thm 4.6}
For any point $\mathbf{w}\in W(d_1,d_2,c)$ with $c\ge0$, and $(d_1,d_2)\ne(1,1)$, $(d_1,d_2)\ne(1,2)$, one has
\begin{equation}\label{dimTwM}
\dim\mathrm{Ext}^1(E(\mathbf{w}),E(\mathbf{w})) = \dim W(d_1,d_2,c).
\end{equation} 
In addition, the morphism $f|_{W(d_1,d_2,c)}$ is an open embedding. Thus
\begin{equation}\label{calm(...)}
\mathcal{C}(d_1,d_2,c):=f(W(d_1,d_2,c))
\end{equation}
is a dense smooth open subset of an irreducible component
$$
\overline{\mathcal{C}(d_1,d_2,c)} = \overline{f(\widetilde{W}(d_1,d_2,c))}
$$ 
of $\call(c+d_1d_2)$, hence also of $\calm(d_1d_2+c)$.
\end{theorem}

\begin{proof}
Equality (\ref{dimTwM}) follows by comparing formulas (\ref{dim W}) and (\ref{dim Ext1}) and using (\ref{d,pa,H})-(\ref{hiN}).

As for the second claim, note that, $f|_{W(d_1,d_2,c)}$ is an injective morphism by construction. Its Kodaira-Spencer map  
$df|_{\mathbf{w}}:\ T_{\mathbf{w}}W(d_1,d_2,c)\to T_{f(\mathbf{w})}\mathcal{M}(d_1d_2+c)=
\mathrm{Ext}^1(E(\mathbf{w}),E(\mathbf{w}))$
is an isomorphism by (\ref{dimTwM}) for any $\mathbf{w}\in W(d_1,d_2,c)$. The assertion follows.
\end{proof}

\begin{remark}
The cases $(d_1,d_2)=(1,1)$, and $(d_1,d_2)=(1,2)$, in which $C$ is either a line or a conic, respectively, were studied in \cite{JMT}, where elementary transformations of instanton bundles by smooth rational curves of arbitrary degree are considered. In fact, for $k=1,2$, $\overline{\mathcal{C}(1,k,c)}$ coincides with the variety $\overline{\cald(k,c+k)}$ introduced in \cite[Section 6]{JMT}. It turns out that $\overline{\cald(k,c+k)}$ are irreducible projective varieties of dimension $8(c+k)-4$ lying in the closure $\overline{\cali(c+k)}$ of $\cali(c+k)$ within $\calm(c+k)$. In particular, for $k=1,2$, $\overline{\mathcal{C}(1,k,c)}$ does not define a new irreducible component of $\calm(c+k)$, diferently from the conclusion of Theorem \ref{Thm 4.6}.
\end{remark}

We conclude this section with a result in the same spirit of \cite[Proposition 3.6]{Ein} and Theorem \ref{ein 0d} above, showing that the the number of irreducible components of $\calm(n)$ whose generic points correspond to sheaves with 1-dimensional singularities becomes arbitrarily large as $n$ grows.

\begin{theorem}\label{Thm 8.7}
Let $\eta_n$ denote the number of irreducible components of $\calm(n)$ whose generic points correspond to sheaves with 1-dimensional singularities. Then \linebreak $\lim_{n\to\infty}\eta_n=\infty$.
\end{theorem}

\begin{proof}
Indeed, the number of ways in which a given positive integer $n$ can be represented as a sum $n=c+d_1d_2$, where $c,\ d_1$ and $d_2$ are positive integers and $(d_1,d_2)\ne(1,1),\ (d_1,d_2)\ne(1,2)$, is unbounded as $n$ grows. Thus the result follows from Theorem \ref{Thm 4.6}.
\end{proof}

\section{Nonemptyness of the intersection $\overline{\mathcal{C}(d_1,d_2,c)}\cap\overline{\cali(n)}$} \label{intersection}

In this section we will perform an inductive procedure showing that each irreducible component
$\overline{\mathcal{C}(d_1,d_2,c)}$ of $\calm(d_1d_2+c)$ (with $c\ge0$, and $(d_1,d_2)\ne(1,1)$, $(d_1,d_2)\ne(1,2)$), constructed in Theorem \ref{Thm 4.6}, has a nonempty intersection with the closure of the instanton component $\overline{\cali(d_1d_2+c)}$.

We start with a construction of a reduced curve $C\in\calh=\calh_{d_1,d_2}$ completely decomposable into
a union of projective lines and satisfying some additional properties. Namely, we prove the following lemma.

\begin{lemma}\label{Lemma 5.1}
For any $1\le d_1\le d_2$ there exists a curve $C\in\calh$ which is completely decomposable into a union of $d_1d_2$ projective lines
\begin{equation}\label{C decomposable}
C=\underset{i=1}{\overset{d_1d_2}{\bigcup}}\ell_i
\end{equation}
such that
\begin{itemize}
\item[(i)] The $\calo_C$-sheaf
\begin{equation}\label{polystable}
L_0=\underset{i=1}{\overset{d_1d_2}{\bigoplus}} \calo_{\ell_i}(-1)
\end{equation}
is $\op3(1)|_C$-semistable.
\item[(ii)] There exists a curve $Y$ with a marked point $0\in Y$, a morphism $f:Y\to\calh$ and a sheaf $\mathbf{L}$ on
$Y\times_{\calh}\calz$, flat over $Y$ and such that, for $C_t:=pr^{-1}(f(t)),\ t\in Y$, where $pr:\calz\to\calh$ is the projection, one has
\begin{equation}\label{isom L0}
\mathbf{L}|_{C_0}\simeq L_0,
\end{equation}
where $C_0$ is the curve $C$ from (\ref{C decomposable}), and 
\begin{equation}\label{Lt}
\mathbf{L}|_{C_t}\ \textit{is locally free},\ \ \ t\in Y^*:=Y\setminus\{0\}.
\end{equation}
\end{itemize}\end{lemma}
\begin{proof}
Let $x_i=(\wp_{i1},...,\wp_{id_i})\in ({\mathbb{P}}^{3\vee})^{\times d_i},\ \ \ i=1,2,$ be two collections of hyperplanes in $\p3$, and set $\ell_{ij}:=\wp^2_{i1}\cap\wp_{2j}$, $\ell_{ij}:=\wp_{i1}\cap \wp_{j2},\ 1\le i\le d_1,\ 1\le j\le d_2$. 
It is clear that, for a general choice of the points $x_i\in({\mathbb{P}}^{3\vee})^{\times d_i},\  i=1,2$, the curve
$C=\underset{i,j=1}{\overset{d_1,d_2}{\cup}}\ell_{ij}$ satisfies the statement of Lemma. We re-enumerate the lines $\ell_{ij}$ as the lines $\ell_i$ in (\ref{C decomposable}).

Next, the sheaf $L_0$ in (\ref{polystable}) is polystable as a direct sum of stable $\op3(1)$-sheaves $\calo_{\ell_i}(-1)$. Hence, it is $\op3(1)|_C$-semistable.

For the second item, set $C_1:=\ell_1$ and, for $2\le k\le d_1d_2-1$, consider a sub-curve
$C_k:=\underset{i=1}{\overset{k}{\cup}}\ell_i$ of $C$ and a subscheme $D_k:=\ell_{k}\cap C_{k-1}$ of $\ell_{k}$. Since
$C\in\calh$, it follows that $D_k$ is a reduced divisor on $C_k$; let, say, $D_k=a_{k1}+...+a_{km_k}$. Consider a sequence of 
$\calo_{C_k}$-sheaves $L_{k}$, where $L_{1}:=\calo_{\ell_1}(-1)$ and for $2\le k\le d_1d_2-1$ the sheaf 
$L_k$ is defined inductively as an extension
\begin{equation}\label{Lk}
0\to L_{k-1}\to L_k\to\calo_{\ell_k}(-1)\to0.
\end{equation}
Each such extension is given by an element of the group
$\mathrm{Ext}^1(\calo_{\ell_k}(-1),L_{k-1})$,
and an easy calculation (cf. \cite[proof of Lemma 4]{Esteves}) shows that
\begin{equation}
\mathrm{Ext}^1(\calo_{\ell_k}(-1),L_{k-1})= 
H^0(\calo_{D_k})=\underset{i=1}{\overset{m_k}{\bigoplus}}\mathbf{k}_{a_{ki}}\simeq\mathbb{A}^{m_k}.
\end{equation}
Furthermore, for any point 
$\tau_k\in\mathbb{A}^{m_k*}:=\{(t_1,...,t_{m_k})\in \mathbb{A}^{m_k}\ |\ t_j\ne0,\ j=1,..., m_k\}$, 
the sheaf $L_k$ is locally free at the points of $D_k$. Hence the last extension $L_{d_1d_2}=L_{d_1d_2}(\mathbf{t})$ 
in the sequence (\ref{Lk}) defined by the element
$\mathbf{t}=(\tau_1,...,\tau_{d_1d_2})\in\mathbf{A}^*:= \mathbb{A}^{m_1*}\times...\times\mathbb{A}^{m_{d_1d_2}*}$
is a locally free $\calo_C$-sheaf. In addition, from (\ref{Lk}) it follows immediately that, for this point $\mathbf{t}$, the sequence of sheaves $0\subset L_1\subset L_2\subset...\subset L_{d_1d_2}(\mathbf{t})$ is a Jordan-H\"older filtration
of $L_{d_1d_2}(\mathbf{t})$ with the associated graded sheaf $L_0$ in (\ref{polystable}). Thus,
$$
[L_0]_S=[L_{d_1d_2}(\mathbf{t})]_S.
$$
Hence, since $L_{d_1d_2}(\mathbf{t})$ is a locally free
$\calo_C$-sheaf, it follows that, in the notation of 
Lemma \ref{Lemma 4.1},
\begin{equation}
[L_0]_S\in\mathbb{M}^0.
\end{equation}

Now, recall the construction of the moduli space $\mathbb{M}$ containing $\mathbb{M}^0$ as a closed subscheme (see, e. g., \cite[Thm 4.3.7]{HL}). Namely, $\mathbb{M}$ is obtained as a GIT-quotient $p:R\to R/\!\!\!\!/GL(N)=\mathbb{M}$ for an appropriately chosen open subset $R$ of the Quot-scheme $\mathrm{Quot}_{\calz/\calh}(\mathcal{V},P)$, where $P$ is the Hilbert polynomial defined in (\ref{hilb poly C}), $\mathcal{V}=\calo_{\calz/\calh}(-m)^{\oplus N},\ N= P(m)$ and $m$ large enough. Now, for the point 
$s=[L_0]_S\in\mathbb{M}^0$, there exists a point 
$[\rho:\mathcal{V}_s\twoheadrightarrow L_0]\in R$ such that $p([\rho])=s$. Consider the closed subscheme 
$R_0=p^{-1}(\mathbb{M}^0)$ of $R$. Since $R_0$ is quasi-projective over $\mathbb{M}^0$ and $\mathbb{M}^s$
is open dense in $\mathbb{M}^0$, it follows that there
exists a curve $Y$ in $R_0$ passing through the point
$0=[\rho]$ and such that
$$
Y^*=Y\setminus \{0\}\subset p^{-1}(\mathbb{M}^s).
$$
By the definition of $\mathbb{M}^s$ the last inclusion 
means that the the universal quotient sheaf on $R_0\times_{\calh}\calz$ being restricted onto
$\calz_Y=Y\times_{\calh}\calz$ becomes a sheaf $\mathbf{L}$ such that, for $t\in Y^*$,   
the sheaf $\mathbf{L}|_{C_t}$ is locally free, where
$C_t=pr^{-1}(f(t))$, and $f=p|_Y:Y\to\calh$ and
$pr:\calz\to\calh$ is the projection. Besides, by the above, $\mathbf{L}|_{C_0}\simeq L_0$.
\end{proof}

\begin{lemma}\label{Lemma 5.1b}
For any $c\ge0$, $1\le d_1\le d_2$, and any $[F]\in\cali(c)$, there exists a curve $C\in\calh$ satisfying the properties of Lemma \ref{Lemma 5.1}, and, in addition, the following ones:
\begin{equation}\label{E trivial}
F|_{\mathbb{P}^1_i}\simeq2\calo_{\mathbb{P}^1_i},\ \ \ 1\le i\le d_1d_2 ~;
\end{equation}
\begin{equation}\label{more Lt}
([F],C_t,[\mathbf{L}|_{C_t}])\in (\cali(c)\times\mathbb{J})_e,\ \ \ t\in Y^*.
\end{equation}
\end{lemma}

\begin{proof}
Both (\ref{E trivial}) and (\ref{more Lt}) are immediate when $c=0$.

Since every instanton bundle of charge $c>0$ is stable, there is, by the Grauert--M\"ullich Theorem (see \cite{B}, \cite{HL}), a divisor $D_{F}$ in the Grassmannian $Gr=Gr(1,\p3)$ such that, for any line $\p1\in Gr\setminus D_{F}$, $F|_{\mathbb{P}^1_i}\simeq2\calo_{\mathbb{P}^1_i}$. Thus, for a general choice of the points $x_i\in({\mathbb{P}}^{3\vee})^{\times d_i},\  i=1,2$,
the condition (\ref{E trivial}) above holds.

From (\ref{E trivial}) and (\ref{polystable}) it follows immediately that 
\begin{equation}\label{property of L0}
([F],C_0,[L:=\mathbf{L}|_{C_0}])\ \textit{satisfies the statements (iii) and (iv) of Lemma \ref{Lemma 4.2}}.
\end{equation}
Since $\mathbf{L}$ is flat over $Y$, the rest is clear by semicontinuity (after possibly shrinking the curve $Y$).
\end{proof}

In the notation of Lemma \ref{Lemma 5.1}, let $p_Y:\calz_Y=\calz\times_{\calh}\calz\to Y$ and
$\boldsymbol{\iota}_Y:\calz\hookrightarrow Y\times\p3$ be the natural projections; let 
$\boldsymbol{F}:=F\boxtimes\calo_Y$, 
$\boldsymbol{\iota}_{Y*}\mathbf{L}(2):=\boldsymbol{\iota}_{Y*}\mathbf{L}\otimes\calo_{\p3}(2)\boxtimes\calo_Y$,
and $\boldsymbol{\tau}_Y:=\inhom_{\calz_Y/Y}(\boldsymbol{F},\boldsymbol{\iota}_{Y*}\mathbf{L}(2))$. 
In addition, consider the projections $\mathbf{p}_Y:{\mathbb{P}}\boldsymbol{\tau}_Y\to Y$, and
$\mathbf{p}_Y:{\mathbb{P}}\boldsymbol{\tau}_Y^*:=\calp\boldsymbol{\tau}_Y\times_YY^*\to Y^*$;
by construction, $\mathbf{p}_Y$ is a projective bundle over $Y$ such that,
\begin{equation}\label{pY}
\mathbf{p}_Y^{-1}(t)=\mathbb{P}(\mathrm{Hom}(F,L_t(2))),\ \ \ 
\mathrm{where}\ L_t:=\mathbf{L}|_{C_t},\ \ \ t\in Y.
\end{equation}
By Lemma \ref{Lemma 5.1b}, one has a morphism
$$
\psi:\ Y^*\to\calb=(\cali(c)\times\mathbb{J})_e,\ \ \ 
t\mapsto([F],C_t,[L_t]),
$$
and from (\ref{fibre of pB}) and (\ref{pY}) it follows
that
\begin{equation}\label{Ptau Y}
\calp\boldsymbol{\tau}_Y^*=\mathbb{P}\boldsymbol{\tau}
\times_{\calb}Y^*.
\end{equation}
Moreover, using Lemma \ref{Lemma 4.2},(iii)-(v), consider
the open dense subset $\widetilde{W}_Y$ of $\calp\boldsymbol{\tau}_Y$ defined as
\begin{equation}\label{tilde WY}
\widetilde{W}_Y:=\{([F],C_t,[L_t],\mathbf{k}\varphi)\ |\ t\in Y,\ \mathbf{k}\varphi\in\mathbb{P}\mathrm{Hom}
(F,C_t,L_t)_e\}.
\end{equation}
Comparing (\ref{tilde WY}) with (\ref{W}) and
using (\ref{Ptau Y}) we obtain the relation
\begin{equation}\label{tilde WY*}
\widetilde{W}_{Y^*}:=\widetilde{W}_{Y}\times_YY^*=
\widetilde{W}(d_1,d_2,c)\times_{\calb}Y^*
\overset{p_W}{\to}\widetilde{W}(d_1,d_2,c).
\end{equation}
On the other hand, consider the morphism
\begin{equation}\label{fY}
f_Y:\ \widetilde{W}_Y\to\call(d_1d_2+c):\ \ \ 
\mathbf{w}=([F],C_t,[L_t],\mathbf{k}\varphi)\mapsto
[E(\mathbf{w})=\ker(\varphi:F\twoheadrightarrow L_t(2))].
\end{equation}
From (\ref{f}) and (\ref{tilde WY*}) it follows that
$f_Y|_{\widetilde{W}_{Y^*}}$ coincides with the composition
$$
\widetilde{W}_{Y^*}
\overset{p_W}{\to}\widetilde{W}(d_1,d_2,c)\overset{f}
{\to}\call(d_1d_2+c).
$$
In view of Theorem \ref{Thm 4.6} this implies that
$f_Y(\widetilde{W}_{Y^*})\subset
\overline{\mathcal{C}(d_1,d_2,c)}$.
Since $\overline{\mathcal{C}(d_1,d_2,c)}$ is projective,
this implies that also
$$
f_Y(\widetilde{W}_Y)\subset\overline{\mathcal{C}(d_1,d_2,c)}.
$$
In particular, since by (\ref{tilde WY})
$$
\mathbb{P}\mathrm{Hom}(F,C_0,L_0)_e=
(\mathbf{p}_Y|_{\widetilde{W}_Y})^{-1}(0),
$$
where $C_0$ is the curve $C$ defined in (\ref{C decomposable}), and $\overline{\mathcal{C}(d_1,d_2,c)}$ is a projective scheme, we obtain the following result.

\begin{theorem}\label{Thm 5.2}
In the conditions and notation of Lemma \ref{Lemma 5.1} and of Lemma \ref{Lemma 5.1b}, there is a morphism
\begin{equation}\label{f0}
\begin{split}
& f_0=f_Y|_{\mathbb{P}\mathrm{Hom}(F,C_0,L_0)_e}: \mathbb{P} \mathrm{Hom}(F,C_0,L_0)_e \to\overline{\mathcal{C}(d_1,d_2,c)}: \\  
& ([F],C_0,[L_0],\mathbf{k}\varphi) \mapsto [\ker(\varphi:F\twoheadrightarrow L_0(2))].
\end{split}
\end{equation}
\end{theorem}

Now, take a point
$$
\mathbf{w} := ([F],C_0,[L_0],\mathbf{k}\varphi)\in\mathbb{P}\mathrm{Hom}(F,C_0,L_0)_e
$$
and, as above, denote
$$
E(\mathbf{w}):=\ker(\varphi:F\twoheadrightarrow L_0).
$$
Next, using (\ref{polystable}), set
$$
L_{(k)}=\underset{i=1}{\overset{k}{\bigoplus}} \calo_{\ell_i}(-1),\ \ \ 1\le k\le d_1d_2,
$$
so that 
\begin{equation}\label{Ld1d2}
L_{(d_1d_2)}=L_0
\end{equation}
and we have splitting exact triples
\begin{equation}\label{triple L(k)}
0\to\calo_{\ell_k}(-1)\to L_{(k)} \overset{\varepsilon_{(k)}}{\to} L_{(k-1)}\to0,\ \ \ 
2\le k\le d_1d_2
\end{equation}
where each $\varepsilon_{(k)}$ is the projection onto a direct summand. Set $\varphi_{(d_1d_2)}:=\varphi$ and, using (\ref{Ld1d2}), define the epimorphisms $\varphi_{(k)}$ as the composition
$$
\varphi_{(k)}:\ F\overset{\varphi}{\to}
L_{(d_1d_2)}(2)\overset{\varepsilon_{d_1d_2}}{\to}
L_{(d_1d_2-1)}(2)\overset{\varepsilon_{d_1d_2-1}}{\to}...
\overset{\varepsilon_{(k)}}{\to}L_{(k)}(2),\ \ \ 
1\le k\le d_1d_2-1. 
$$
Set $E_{(k)}:=\ker\varphi_{(k)}$, so that, by the above,
\begin{equation}\label{Ed1d2=Ew}
E_{(d_1d_2)}=E(\mathbf{w})
\end{equation}
and, for $2\le k\le d_1d_2$, there is a commutative diagram
\begin{equation}\label{diagram2}
\xymatrix{
& & & 0\ar[d] & \\
& 0\ar[d] & & \calo_{\ell_k}(1)\ar[d] & \\
0\ar[r] & E_{(k)}\ar[r]\ar[d] &
F\ar[r]^-{\varphi_{(k)}}\ar@{=}[d] & L_{(k)}(2)\ar[r]\ar[d]^-{\varepsilon_{(k)}} & 0\\
0\ar[r] & E_{(k-1)}\ar[r]\ar[d]^-{\theta} &
F\ar[r]^-{\varphi_{(k-1)}} & L_{(k-1)}(2)\ar[r]\ar[d] & 0\\
& \calo_{\ell_k}(1) \ar[d] & & 0 &\\
& 0, & & &}
\end{equation}
in which the right vertical triple is the triple (\ref{triple L(k)}) twisted by the sheaf $\op3(2)$.

We are going to show, by induction on $k$, that
\begin{equation}\label{[Ek]}
[E_{(k)}]\in\overline{\cali(c+k)},\ \ \ 1\le k\le d_1d_2.
\end{equation}

First, in the case $k=1$ we have the exact triple
\begin{equation}\label{E_1}
0 \to E_{(1)} \to F \to L_{(1)}(2) \to 0
\end{equation}
which by \cite[Proposition 7.2]{JMT} yields (\ref{[Ek]}) for $k=1$.

Next, given $k\ge2$, assume that (\ref{[Ek]}) is true for $k-1$, i. e., in the diagram (\ref{diagram2}),
$$ [E_{(k-1)}]\in\overline{\cali(c+k-1)}. $$
This implies that there exists a curve $T$ with a marked point $0\in T$ and a sheaf $\mathbf{E}'$ on $\p3\times T$, flat over $T$ and such that
\begin{equation}\label{E'0=}
\mathbf{E}'|_{\p3\times\{0\}}\simeq E_{(k-1)}
\end{equation}
and, for $t\in T^*=T\smallsetminus\{0\}$, the sheaf $E'_t:=\mathbf{E}'|_{\p3\times\{t\}}$ is an instanton
bundle from $\cali(c+k-1)$:
\begin{equation}\label{E't}
[E'_t]\in\cali(c+k-1),\ \ \ t\in T^*.
\end{equation}
Now one easily sees that, after possibly shrinking the curve $T$, the epimorphism $\theta$ in the diagram (\ref{diagram2}) extends to an epimorphism $\boldsymbol{\Theta}:\ \mathbf{E}'\twoheadrightarrow\calo_{\ell_k}(1)\boxtimes\calo_T$, and we denote
$\mathbf{E}=\ker\boldsymbol{\Theta}$. By construction, an $\calo_{\mathbb{P}^3\times T}$-sheaf $\mathbf{E}$ is flat over $T$, so that, restricting the exact triple $0 \to \mathbf{E} \to \mathbf{E}' \to \calo_{\ell_k}(1)\boxtimes\calo_T \to 0$ onto $\p3\times\{t\},
\ t\in T,$ we obtain an exact triple 
\begin{equation}\label{triple Et}
0\to E_t\to E'_t\to\calo_{\ell_k}(1)\to0, 
\end{equation}
where $E_t=\mathbf{E}|_{\p3\times\{t\}}$. By \cite[Proposition 7.2]{JMT}, this triple together with (\ref{E't}) implies that 
\begin{equation}\label{Et in}
[E_t]\in\overline{\cali(c+k)},\ \ \ t\in T^*. 
\end{equation}
On the other hand, by the construction, the triple (\ref{triple Et}) for $t=0$ coincides with the left vertical triple in
(\ref{diagram2}), so that
\begin{equation}\label{E0=}
E_0\simeq E_{(k)}.
\end{equation}
Besides, in the case $c>0$, since $F$ is $\mu$-stable, the upper horizontal triple in diagram (\ref{diagram2}) easily shows that the sheaf $E_0$ is ($\mu$-)stable as well. When $c=0$ and $F=2\cdot\op3$ in diagram (\ref{diagram2}), we again proceed by induction on
$k$. For $k=1$, triple (\ref{E_1}) and \cite[Lemma 4.3]{JMT} implies that $E_{(1)}$ is stable. Now assume that $E_{(k-1)}$ is stable; the first column of diagram (\ref{diagram2}) immediately implies that $E_{(k)}$ is also stable, since any sheaf that would destabilize $E_{(k)}$ would also destabilize $E_{(k-1)}$.

Thus, in view of (\ref{Et in}) and (\ref{E0=}) we obtain a modular morphism 
$$ f:T\to\calm(c+k),\ t\mapsto[E_t]. $$
Since $\overline{\cali(c+k)}$ is closed in $\calm(c+k)$, it follows that $[E_{(k)}]=f(0)\in\overline{\cali(c+k)}$, i. e. we obtain (\ref{[Ek]}).

In particular, (\ref{Ed1d2=Ew}) and (\ref{[Ek]}) yield
$$ [E(\mathbf{w})]\in\overline{\cali(c+d_1d_2)}. $$
Since by construction $[E(\mathbf{w})]\in\im f_0$, it follows from Theorem \ref{Thm 5.2} that
$[E(\mathbf{w})]\in\overline{\cali(c+d_1d_2)}\cap \overline{\mathcal{C}(d_1,d_2,c)}$. 

Finally, note that each $E_{(k)}$, and hence $E_0$, is actually an instanton sheaf. Indeed, since $F$ is a locally free instanton sheaf, one easily checks from triple (\ref{E_1}) that $E_{(1)}$ is an instanton sheaf. Assuming that $E_{(k-1)}$ is an instanton sheaf, one can use the first column of diagram (\ref{diagram2}) to check that so is $E_{(k)}$.

Summing it all up, we obtain the following theorem.

\begin{theorem} \label{intersection 1d}
For any $c\ge0$ and $1\le d_1\le d_2$, 
$$ \overline{\cali(c+d_1d_2)}\cap \overline{\mathcal{C}(d_1,d_2,c)}\ne\emptyset. $$
In addition, the above intersection contains instanton sheaves.
\end{theorem}

\section{Intersection of $\overline{\mathcal{C}(d_1,d_2,c)}$ with the Ein components}\label{n cap c}

For any three integers $c>b\ge a\ge0$, consider the monad
\begin{equation}\label{ein monads}
0\to \op3(-c) \stackrel{\alpha}{\longrightarrow} \op3(-b) \oplus \op3(-a) \oplus \op3(a) \oplus \op3(b)
\stackrel{\beta}{\longrightarrow} \op3(c) \to 0 ,
\end{equation}
with morphisms given by
\begin{equation}\label{alpha beta}
\alpha = \left( \begin{array}{c} \sigma_4 \\ \sigma_3 \\ -\sigma_2 \\ -\sigma_1 \\\end{array} \right)
~~{\rm and}~~ \beta=
\left( \begin{array}{cccc} \sigma_1 & \sigma_2 & \sigma_3 & \sigma_4 \end{array} \right)
\end{equation}
where
$$ \sigma_1\in H^0(\op3(c+b)) ~~,~~ \sigma_2\in H^0(\op3(c+a)) $$
$$ \sigma_3\in H^0(\op3(c-a)) ~~,~~ \sigma_4\in H^0(\op3(c-b)) $$
do not vanish simultaneously. Ein showed in \cite[Proposition 1.2(a)]{Ein} that the cohomology of such a monad is stable if and only if $c>a+b$; in this case, there exists an irreducible component $\caln(a,b,c)$ of $\calb(c^2-b^2-a^2)$ whose generic point corresponds to a locally free sheaf given as cohomology of (\ref{ein monads}). Such components are called \emph{Ein components}.

Let $\overline{\caln(a,b,c)}$ denote the closure of $\caln(a,b,c)$ within $\calm(c^2-b^2-a^2)$. The main goal of this section is to establish the following result.

\begin{proposition} \label{N inter C}
The components $\overline{\caln(0,b,c)}$ and $\overline{\mathcal{C}(c-b,c+b,0)}$ intersect within $\calm(c^2-b^2)$ along a subvariety of codimension $1+(c^2-b^2)(c-2)$ in $\overline{\mathcal{C}(c-b,c+b,0)}$.
\end{proposition}

\begin{proof} 
Given a parameter $t\in\mathbb{A}^1$, consider the following family of monads
\begin{equation}\label{family monad}
0\to \op3(-c) \stackrel{\alpha_t}{\rightarrow} \op3(-b) \oplus 2\cdot\op3 \oplus \op3(b)
\stackrel{\beta}{\rightarrow} \op3(c) \to 0,
\end{equation}
where $\alpha_t$ is given by
$$ \alpha_t = \left( \begin{array}{c} \sigma_4 \\ t\cdot\sigma_3 \\ -t\cdot\sigma_2 \\ -\sigma_1 \end{array} \right) $$
while $\beta$ is given as in (\ref{alpha beta}). Clearly, for each $t\ne0$, the sheaf $E_t=:\ker\beta/\im\alpha_t$ defines a point $[E_t]\in\caln(0,b,c)$; we thus obtain a modular morphism
$$ \mathbb{A}^1\setminus\{0\} \to \calm(c^2-b^2),  ~~~~~~ t \mapsto [E_t] $$
whose image lies within $\caln(0,b,c)$.

Next, we show that $E_0$ fits into the following exact triple:
\begin{equation} \label{e0 sqc}
0 \to E_0 \to 2\cdot \op3 \to \calo_{\Gamma}(c) \to 0 ,
\end{equation}
where $\Gamma$ is the complete intersection curve defined by $\{\sigma_1=\sigma_4=0\}$; we may assume that $\Gamma$ is irreducible since this is an open condition. Note also that $\calo_{\Gamma}(c-2)$ is a theta-characteristic on $\Gamma$.

Indeed, consider the following short exact sequence of complexes:
$$ \xymatrix{
& 0 \ar[d] & 0 \ar[d] & 0 \ar[d] & \\
0 \ar[r] & \op3(-c) \ar[d]^{\simeq} \ar[r]^{\tilde{\alpha}} & \op3(-b) \oplus \op3(b) \ar[d] \ar[r]^{\tilde{\beta}} & \op3(c) \ar[d]^{\simeq} \ar[r] & 0 \\
0 \ar[r] & \op3(-c) \ar[d] \ar[r]^{\!\!\!\!\!\!\!\!\!\!\!\!\!\!\!\!\!\!\!\!\!\!\!\!\!\!\!\!\alpha_0} &
\op3(-b) \oplus 2\cdot \op3 \oplus \op3(b) \ar[d] \ar[r]^{~~~~~~~~~~~~\beta} &
\op3(c) \ar[d] \ar[r] & 0 \\
& 0 \ar[d] \ar[r] & 2\cdot \op3 \ar[d] \ar[r] & 0 \ar[d] & \\
& 0 & 0 & 0 &
} $$
where the complex in the middle line is (\ref{family monad}) for $t=0$, and the morphisms $\tilde{\alpha}$ and $\tilde{\beta}$ are given by
$$ \tilde{\alpha} = \left( \begin{array}{c} -\sigma_4 \\ \sigma_1 \\\end{array} \right)
~~{\rm and}~~ \beta= \left( \begin{array}{cccc} \sigma_1 & \sigma_4 \end{array} \right). $$
Passing to cohomology, we obtain precisely the triple (\ref{e0 sqc}).

By \cite[Lemma 4.3]{JMT}, (\ref{e0 sqc}) implies that $E_0$ is stable, hence it follows that $[E_0]\in\overline{\caln(0,b,c)}$. On the other hand, one clearly sees from (\ref{e0 sqc}) and from the definition of $\calc(c-b,c+b,0)$ in Theorem \ref{Thm 4.6} that $[E_0]\in\overline{\calc(c-b,c+b,0)}$. 

Finally, note that a generic point $[E]$ in the intersection
$$ \overline{\caln(0,b,c)}~\cap~\overline{\calc(c-b,c+b,0)} $$
will be precisely of the form (\ref{e0 sqc}), fixing the choice of a line bundle in $\mathrm{Pic}^{g-1}(\Gamma)$, where $g=1+(c^2-b^2)(c-2)$ is the genus of $\Gamma$. It is then easy to see that such sheaves form a family of codimension $g$ in $\overline{\calc(c-b,c+b,0)}$.
\end{proof}

\section{Connectedness of $\calm(2)$}\label{M(2)}

As mentioned at the Introduction, it is not difficult to check that $\calm(1)$ is irreducible; this fact is probably well known to specialists, but for lack of a suitable reference, we present a brief argument here.

The key point is to show that every semistable rank 2 sheaf $E$ on $\p3$ with $c_1(E)=0$, $c_2(E)=1$ and $c_3(E)=0$ is a nullcorrelation sheaf in the sense of \cite{Ein82}, that is, given by an exact sequence of the form
$$ 0 \to \op3(-1) \stackrel{\sigma}{\longrightarrow} \Omega^1_{\p3}(1) \to E \to 0. $$
It follows that $E$ is uniquely determine by the section $\sigma \in H^0(\Omega^1_{\p3}(2))$ up to scalar multiples, so that $\calm(1)\simeq\mathbb{P}H^0(\Omega^1_{\p3}(2))$.

Indeed, semistability implies that $h^0(E)=0$. If $E$ is locally free, then $E$ is stable and it follows from Barth's theory of spectra, see \cite{B2} or \cite[Section 7]{H-r}, that $E$ is an instanton bundle of charge 1, and these are precisely the locally free nullcorrelation sheaves. On the other hand, if $E$ is not locally free, then $E^{\vee\vee}$ is a $\mu$-semistable rank 2 reflexive sheaf with $c_1(E^{\vee\vee})=0$ and $c_2(E^{\vee\vee})=0,1$. If $c_2(E^{\vee\vee})=1$, so that $E$ has 0-dimensional singularities, then $E^{\vee\vee}$ must be stable by Lemma \ref{stability-lemma}; it follows from \cite[Lemma 2.1]{Chang} that $c_3(E^{\vee\vee})=0$, so $E^{\vee\vee}/E=0$, contradicting the hypothesis of $E$ not being locally free. Therefore we must have $c_2(E^{\vee\vee})=0$, hence $E^{\vee\vee}\simeq2\cdot\op3$ and
$$ 0 \to E \to 2\cdot\op3 \to \calo_\ell(1) \to 0 $$
where $\ell$ is a line. One can then check that $E$ satisfies the cohomological conditions of \cite[Proposition 1.1]{Ein82}, so that $E$ is a nullcorrelation sheaf.

\bigskip

Next, we recall the description of $\calm(2)$ given by Hartshorne \cite{H-vb}, Le Potier \cite{LeP} and Trautmann \cite{Tr}. 
By \cite[Section 9]{H-vb} the scheme $\calb(2)$ coincides with the instanton component $\cali(2)$ of dimension 13, so its closure $\overline{\cali(2)}$ is an irreducible component of $\calm(2)$. According to \cite[Thm. 7.12]{LeP}, $\calm(2)$ contains two additional irreducible components, which are given by the closures of the subschemes
$$ \calp(2)_l = \left\{ [E] \in \calm(2) ~|~ \dim\ext^2(E,\op3)=l \right\} ~~ l=1,2 $$
whithin $\calm(2)$; furthermore, $\dim \overline{\calp(2)_l}=13+4l$. Le Potier calls these the \emph{Trautmann components}.

Note that these actually coincides with the components $\overline{\calt(2,l)}$ described in Section \ref{0dsing} above. Indeed, note that if $[E]\in\calt(2,l)$, then 
$$ \dim\ext^2(E,\op3) = h^0(\inext^2(E,\op3)) = h^0(\inext^3(Q_E,\op3))=h^0(Q_E). $$
However, the length of $Q_E$ is half of $c_3(E^{\vee\vee})$, which means that $[E]\in\calp(2)_l$, thus $\calt(2,l)\subset\calp(2)_l$.

In addition, Chang proved in \cite[Section 2]{Chang} that, for each $l=1,2$, $\calr(0;2;2l)$ is irreducible, nonsingular of dimension 13. It follows from Theorem  \ref{0d comp's} that, for each $l=1,2$,  $\overline{\calt(2,l)}$ is an irreducible component of $\calm(2)$ of dimension $13+4l$; therefore, we must have that $\overline{\calt(2,l)}=\overline{\calp(2)_l}$.

Consequently, Le Potier's result can be restated in the following form, see also \cite{Tr}: 
\begin{equation} \label{full M(2)}
\calm(2)=\overline{\cali(2)}\cup\overline{\calt(2,1)}\cup\overline{\calt(2,2)}.
\end{equation}
The main goal of this section is to show that $\calm(2)$ is connected.

Recall from Section \ref{0dsing} that a generic sheaf $E$ from $\overline{\calt(2,1)}$ is obtained as the kernel of an epimorphism $\epsilon:\ F\twoheadrightarrow\calo_q$ where $F$ is a generic reflexive sheaf from $\calr(0;2;2)$ and
$q\not\in\sing(F)$:
\begin{equation}\label{R(0,2,2)}
0 \to E \to F \overset{\epsilon}{\to} \calo_q \to 0,\ \ \ q\not\in\sing(F).
\end{equation}
Every $[F]\in\calr(0;2;2)$ satisfies $h^0(F(1))=3$, cf. \cite[Table 2.8.1]{Chang}; moreover, the zero scheme $Y=(s)_0$ of a generic section $s\in H^0(F(1))$ is a disjoint union of a line $\ell$ and a nonsingular conic $C$ \cite[Lemma 2.7]{Chang}, i. e. there is an exact triple
\begin{equation}\label{F,M1}
0\to \op3(-1) \to F \to I_{Y/\p3}(1)\to 0,\ \ \ Y=\ell\sqcup C. 
\end{equation}

In addition, a generic sheaf $E$ from $\overline{\calt(2,4)}$ is obtained as the kernel of an epimorphism $\epsilon:\ F\twoheadrightarrow \calo_{q_1}\oplus\calo_{q_2}$ where $F$ is a generic reflexive sheaf from $\calr(0;2;4)$ with $q_1,q_2\not\in\sing(F)$, and $q_1\ne q_2$:
\begin{equation}\label{R(0,2,4)}
0\to E\to F\overset{\epsilon}{\to} \calo_{q_1}\oplus\calo_{q_2}\to0,\ \ \ q_1,q_2\not\in \sing(F),\ \ q_1\ne q_2.
\end{equation}
Every $[F]\in\calr(0;2;4)$ satisfies $h^0(F(1))=4$, cf. \cite[Table 2.12.2]{Chang}; the zero scheme $Y=(s)_0$ of a generic section $s\in H^0(F(1))$ is a nonsingular twisted cubic curve \cite[Lemma 2.13]{Chang}, i. e. there is an exact triple
\begin{equation}\label{F,M2}
0\to\op3(-1)\to E^{\vee\vee}\to I_{Y/\p3}(1)\to 0. 
\end{equation}

We are finally in position to prove the main result of this section.

\begin{theorem} \label{T inter I 2}
Both components $\overline{\calt(2,1)}$ and $\overline{\calt(2,2)}$ have nonempty intersection with the instanton component $\overline{\cali(2)}$. In particular, $\calm(2)$ is connected.
\end{theorem}

\begin{proof}
We first consider the case $l=1$. Recall from \cite{H-vb} that a generic locally free sheaf $G$ from $\cali(2)$ is a \emph{'t Hooft bundle}, fitting in an exact triple
\begin{equation}\label{E,I2}
0\to \op3(-1) \to G \to I_{Z/\p3}(1)\to 0,\ \ \ Z=\ell_0\sqcup \ell_1\sqcup \ell_2,
\end{equation}
where $\ell_0,\ell_1,\ell_2$ are disjoint lines in $\p3$. We include $Z$ as a generic fibre $Z_t,\ t\ne0,$ into a 1-dimensional flat family $\calz$ of curves in $\p3$:
\begin{equation}\label{calZ}
\pi:\calz\hookrightarrow\p3\times U\overset{pr_2}{\to}U,
\end{equation}
with base $U\ni0$ which is an open subset of $\mathbb{A}^1$, such that
\begin{itemize}
\item[(a)] for $t\ne0$ the fibre $Z_t=\pi^{-1}(t)$ of the family $\calz$ is a disjoint union of three lines in $\p3$;
\item[(b1)] the zeroth fibre $Z_0$ of this family, being reduced, is a union of lines
\begin{equation}\label{Z0red}
(Z_0)_{red}=\ell_0\sqcup(\ell_1\cup \ell_2),\ \ \ w:=\ell_1\cap \ell_2=\{\mathrm{pt}\}.
\end{equation}
and as a scheme $Z_0$ has an embedded point $w$:
\begin{equation}\label{embedded pt}
0\to\calo_w\to\calo_{Z_0}\to\calo_{(Z_0)_{red}}\to0.
\end{equation}
\end{itemize}
The sheaf $G$ from (\ref{E,I2}) can then be included into the family $\boldsymbol{G}$ of sheaves on $\p3$ with base $U$ fitting in the exact triple
\begin{equation}\label{bold calE 1}
0 \to \op3(-1)\boxtimes\calo_U \to \boldsymbol{G} \to 
I_{\calz/\p3\times U}\otimes\calo_{\p3}(1) \boxtimes \calo_U\to 0.
\end{equation} 
We thus obtain a modular morphism
\begin{equation}\label{Phi U}
\Phi_U:\ U\to \calm(2) ,\ t\mapsto[G_t],\ \ \ 
G_t:=\boldsymbol{G}|_{\p3\times\{t\}}.
\end{equation}

Moreover, in view of (\ref{bold calE 1}), the sheaf $G_0$ fits into the exact triple 
$$ 0 \to \op3(-1) \stackrel{r}{\to} G_0\to I_{Z_0/\p3}(1)\to 0; $$
composing the morphism $r$ in the previous equation with the standard monomorphism $G_0 \to G_0^{\vee\vee}$ we obtain, using the triple (\ref{embedded pt}), the following exact triples for $G_0^{\vee\vee}$:
\begin{equation}\label{E refl l=1}
0\to G_0\to G_0^{\vee\vee}\to\calo_q\to 0,
\end{equation}
\begin{equation}\label{E refl triple}
0\to \op3(-1) \overset{s}{\to} G_0^{\vee\vee}\to I_{(Z_0)_{red}/\p3}(1)\to 0.
\end{equation}
Now (\ref{Z0red}) and (\ref{E refl triple}) show that $s$ is a section of a reflexive sheaf $G_0^{\vee\vee}$ having a disjoint union $(Z_0)_{red}$ of a line and a reducible conic as the zero scheme. Note that $G_0^{\vee\vee}$ is $\mu$-stable, since $Y$ is not contained in a plane (cf. \cite[Proposition 4.2]{H-r}); it follows that $G_0$ is stable. Since, by construction, $\Phi(U\setminus\{0\})\subset\cali(2)$, we conclude that 
\begin{equation}\label{E0 limit inst}
[G_0]\in\overline{\cali(2)}.
\end{equation}

By the description of $\calr(0;2;2)$ above, it follows that the triple (\ref{E refl triple}) is a specialization of the triple (\ref{F,M1}) within a flat family of triples in which the conic $C$ specializes into a reducible conic $\ell_1\cup \ell_2$, so that $Y$ in (\ref{F,M1}) specializes to $(Z_0)_{red}$. It follows that
the triple (\ref{E refl l=1}) is a flat specialization of the triple (\ref{R(0,2,2)}), so that
\begin{equation}\label{E0 limit M1}
[G_0]\in\overline{\calt(2,1)}.
\end{equation}
Finally, (\ref{E0 limit inst}) and (\ref{E0 limit M1}) imply that $\overline{\calt(2,1)}\cap\overline{\cali(2)}\ne\emptyset$, as desired.

Next, consider the case $l=2$; one takes a family $\calz$ as in (\ref{calZ}) satisfying  property (a) above, and replacing property (b1) by the following one:
\begin{itemize}
\item[(b2)] the zeroth fibre $Z_0$ of this family, being reduced, is a (connected)
chain of three lines not lying in a plane:
\begin{equation}\label{Z0red 1}
(Z_0)_{red}=\ell_0\cup \ell_1\cup \ell_2 ,\ \ \ q_1:=\ell_0\cap \ell_1=\{\mathrm{pt}\},\
q_2:=\ell_1\cap \ell_2=\{\mathrm{pt}\},\  q_1\ne q_2.
\end{equation}
and as a scheme $Z_0$ has two embedded points $q_1$ and $q_2$:
$$ 0\to\calo_{q_1}\oplus\calo_{q_2}\to\calo_{Z_0}\to\calo_{(Z_0)_{red}}\to0. $$
\end{itemize}
Then, as above, the 't Hooft bundle $G$ from (\ref{E,I2}) is included into the family $\boldsymbol{G}$ of sheaves on $\p3$ given by the exact triple (\ref{bold calE 1}). In this case, instead of the triple (\ref{E refl l=1}), one has an exact triple
\begin{equation}\label{E refl 1}
0\to G_0\to G_0^{\vee\vee} \to \calo_{q_1}\oplus\calo_{q_2} \to 0,
\end{equation}
Besides, the triple (\ref{E refl triple}) holds as before; thus, in view of (\ref{Z0red 1}), the morphism $s$ in (\ref{E refl triple}) is a section of a reflexive sheaf $G_0^{\vee\vee}(1)$ having the chain of lines $(Z_0)_{red}$ in (\ref{Z0red 1}) as its zero scheme. Note that $G_0^{\vee\vee}$ is $\mu$-stable, since $Y$ is not contained in a plane (cf. \cite[Proposition 4.2]{H-r}); it also follows that $G_0$ is stable, so that $[G_0]\in\overline{\cali(2)}$, where, as before, $G_0=:\boldsymbol{G}|_{\p3\times\{0\}}$.

Hence, from the above description of $\calr(0;2;4)$, it follows that the triple (\ref{E refl triple}) is a specialization of the triple (\ref{F,M2}) within a flat family of triples in which the twisted cubic $Y$ specializes to the chain of lines $(Z_0)_{red}$ in (\ref{Z0red 1}). It follows that the triple
(\ref{E refl 1}) is a flat specialization of the triple (\ref{R(0,2,4)}), so that
\begin{equation}\label{E0 limit M2}
[G_0]\in[G_0]\in\overline{\calt(2,2)}.
\end{equation}
Now $\overline{\calt(2,2)}\cap\overline{\cali(2)}\ne\emptyset$  follows from (\ref{E0 limit inst}) and
(\ref{E0 limit M2}).
\end{proof}

\begin{remark}\rm
It follows from \cite[Theorem 7.8]{JMT}, (\ref{full M(2)}) and Theorem \ref{T inter I 2} that the boundary of charge 2 instanton bundles
$$ \partial\cali(2) := \overline{\cali(2)}\setminus\cali(2) $$
has exactly four components, divided into two types:
\begin{itemize}
\item[(I)] $\overline{\calc(1,1,1)}$ and $\overline{\calc(1,2,0)}$, which corresponds to $\cald(1,2)$ and $\cald(2,2)$, respectively, in the notation of \cite[Theorem 7.8]{JMT}; and
\item[(II)] $\overline{\calt(2,l)}\cap\overline{\cali(2)}$ for $l=1,2$.
\end{itemize}
Indeed, either $[E]\in \partial\cali(2)$ is an instanton sheaf, so $[E]$ lies in one of the components of type (I), or $[E]$ is not an instanton sheaf, in which case $[E]$ lies in one of the components of type (II). In addition, by \cite[Proposition 6.4]{JMT}, the components of type (I) are irreducible and divisorial. 

The fact that $\partial\cali(2)$ has exactly 4 components was first observed by Narasimhan and Trautmann in \cite{NT}; they also showed that all four components are irreducible and divisorial. Therefore, it is quite reasonable to conjecture that the components of type (II) above are both irreducible and divisorial. 
\end{remark}

\section{Irreducible components of $\calm(3)$} \label{M(3)}

Ellingsrud and Stromme showed in \cite{ES} that $\calb(3)$ has precisely two irreducible components, both nonsingular, rational and of the expected dimension 21; these can be described as follows:
\begin{itemize}
\item the \emph{instanton component} $\cali(3)$, whose points are the cohomology of monads of the form
$$ 0\to 3\cdot\op3(-1) \to 8\cdot \op3 \to 3\cdot\op3(1) \to0; 
$$
\item the \emph{Ein component} $\caln(0,1,2)$, following the notation of Section \ref{n cap c} above, whose points are the cohomology of monads of the form
$$ 0\to \op3(-2) \to \op3(-1) \oplus 2\cdot\op3 \oplus \op3(1) \to \op3(2) \to0. $$
\end{itemize}

Recall also that Chang \cite[Section 3]{Chang} proved that $\calr(0;3;l)$ is, for each $l=1,\dots,4$, irreducible and of expected dimension 21; in addition, $\calr(0;3;4)$ and $\calr(0;3;8)$ are rational, while $\calr(0;3;6)$ is unirational. Therefore, we can apply Theorem \ref{0d comp's} to show that there are four irreducible components $\overline{\calt(3,l)}$ of dimensions $21+4l$, for each $l=1,\dots,4$ within $\calm(3)$.

Furthermore, Theorem \ref{Thm 4.6} provides one additional irreducible component whose generic point corresponds to sheaves with 1-dimensional singularities, labeled $\overline{\mathcal{C}(1,3,0)}$ in Section \ref{1dsing}.

We therefore conclude that $\calm(3)$ has at least seven irreducible components, divided into 3 types, as below:
\begin{itemize}
\item[(I)] $\overline{\cali(3)}$ and $\overline{\caln(0,1,2)}$, both of dimension 21, and whose generic point corresponds to a locally free sheaf; 
\item[(II)] $\overline{\mathcal{C}(1,3,0)}$, of dimension 21; whose generic point corresponds to a sheaf which is singular along smooth plane cubic;
\item[(III)] $\overline{\calt(3,l)}$ for $l=1,2,3,4$, of dimension $21+4l$, whose generic point corresponds to a sheaf which is singular along $3l$ distinct points.
\end{itemize}

In this section, we prove the following.

\begin{theorem}\label{connected union}
The union
$$ \overline{\cali(3)}\cup\overline{\caln(0,1,2)}
\cup\overline{\mathcal{C}(1,3,0)}\cup\overline{\calt(3,1)}\cup
\overline{\calt(3,2)}\cup\overline{\calt(3,3)}\cup
\overline{\calt(3,4)} $$
is connected.
\end{theorem}

\begin{remark}\rm
We conjecture that $\calm(3)$ has no other irreducible components, and hence it is connected. Our list certainly exhausts all irreducible components of $\calm(3)$ whose generic point corresponds either to a locally free sheaf, or to a sheaf with 0-dimensional singularities. However, it is not clear to us at the moment whether $\calm(3)$ possesses other irreducible components whose generic point corresponds to a sheaf with either 1-dimensional or mixed singularities.
\end{remark}

First, note that Theorem \ref{intersection 1d} guarantees, in particular, that 
$\overline{\cali(3)}\cap\overline{\mathcal{C}(1,3,0)}\ne\emptyset$,
something that has also been been remarked by Perrin in \cite[Thm 0.1]{Per3}. In addition, Proposition
\ref{N inter C} for $c=2$ and $b=1$ guarantees that $\overline{\mathcal{C}(1,3,0)}$ also intersects the closure of $\caln(0,1,2)$. Therefore, the proof of Theorem \ref{connected union} is completed by proving the following result.

\begin{proposition}\label{T inter I}
For each $l=1,2,3,4$, we have $\overline{\calt(3,l)}\cap\overline{\cali(3)}\ne\emptyset$.
\end{proposition}

\begin{proof}
We begin by recalling the description of $\overline{\calt(3,l)},\ l=1,2,3,4,$ from Section \ref{0dsing}. A generic sheaf $E$ from $\overline{\calt(3,l)}$ is defined as the kernel of an epimorphism
$\epsilon:\ F\twoheadrightarrow \oplus_{i=1}^l\calo_{q_j}$, where $[F]\in\calr(0;3;2l)$
$\calr(0;3;2l)$ 
and $q_j\not\in\sing(F),\ j=1,...,l$:
\begin{equation}\label{R(0,3,l)}
0\to E\to F \overset{\epsilon}{\to} \oplus_{j=1}^l \calo_{q_j}\to 0,
\ \ \ q_j\not\in\sing(F),\ j=1,...,l.
\end{equation}
According to Chang \cite[Section 3]{Chang}, for each $l=1,...,l$ there exist sheaves $F$ in $\calr(0;3;l)$ such that $h^0(F(1))>0$, and whose zero scheme $Y=(s)_0$ of a nontrivial section $s\in H^0(F(1))$ can be described as follows:
\begin{itemize}
\item[(i)] For $l=1$, the scheme $Y$ is a disjoint union $\ell_1\sqcup \ell_2\sqcup C$ of two lines $\ell_1,\ell_2$ and a nonsingular conic $C$;
\item[(ii)] for $l=2$, the scheme $Y$ is a disjoint union $\ell\sqcup C$ of a line $\ell$ and a nonsingular twisted cubic $C$, cf. \cite[proof of Thm. 3.4]{Chang};
\item[(iii)] for $l=3$, the scheme $Y$ is a nonsingular rational quartic curve, cf. \cite[proof of Thm. 3.5]{Chang};
\item[(iv)] for $l=4$, the scheme $Y$ is a nonsingular space elliptic quartic curve, cf. \cite[proof of Lemma 3.8]{Chang}.
\end{itemize}
In addition, such sheaves are generic for $l=2,3,4$, but special in the case $l=1$.

Let us first consider the case $l=1$. We repeat with minor modifications the argument from Section \ref{M(2)}; more precisely, compare with equations (\ref{R(0,2,2)}), (\ref{F,M2})-(\ref{E0 limit M1}). First,
similar to (\ref{F,M1}) the sheaf $F=E^{\vee\vee}$ fits in an exact triple
\begin{equation}\label{F,T1}
0\to\op3(-1)\to E^{\vee\vee}\to I_{Y/\p3}(1)\to 0,\ \ \ Y=\ell_1\sqcup \ell_2\sqcup C. 
\end{equation}
Next, according to \cite{H-vb} there exists a 't Hooft vector bundle $H$ in $\cali(3)$ given by an exact triple
\begin{equation}\label{E,I3}
0\to\op3(-1)\to H \to I_{Z/\p3}(1)\to 0,\ \ \ Z= \ell_1\sqcup \ell_2\sqcup \ell_3\sqcup \ell_4,
\end{equation}
where $\ell_1,...,\ell_4$ are disjoint lines in $\p3$ of which $\ell_1$ and $\ell_2$ are taken from (i) above. We now include $Z$ as a generic fibre $Z_t,\ t\ne0,$ into a 1-dimensional flat family $\calz$ of curves in $\p3$ as in (\ref{calZ}),
with base $U\ni0$ being an open subset of $\mathbb{A}^1$, such that
\begin{itemize}
\item [(a)] for $t\ne0$ the fibre $Z_t=\pi^{-1}(t)$ of the family $\calz$ is a disjoint union of four lines in $\p3$;
\item [(b)] the zeroth fibre $Z_0$ of this family, being reduced, is a union of lines
\begin{equation}\label{Z0red l=1}
(Z_0)_{red} = \ell_1\sqcup \ell_2\sqcup(\ell_3\cup \ell_4),\ \ \ q=\ell_3\cap \ell_4=\{\mathrm{pt}\}.
\end{equation}
and as a scheme $Z_0$ has an embedded point $q_1$:
\begin{equation}\label{again embedded pt}
0\to \calo_{q} \to \calo_{Z_0} \to \calo_{(Z_0)_{red}} \to0.
\end{equation}
\end{itemize}

The sheaf $H$ defined in (\ref{E,I3}) is then included into the family $\boldsymbol{H}$ of sheaves on $\p3$ with base $U$ given by the exact triple
\begin{equation}\label{bold calE}
0\to \op3(-1)\boxtimes\calo_U \to \boldsymbol{H} \to
I_{\calz/\p3\times U}\otimes\calo_{\p3}(1)\boxtimes\calo_U \to 0.
\end{equation} 
Thus as in (\ref{Phi U}) we obtain a modular morphism $\Phi_U:\ U\to\overline{\cali(3)},\ t\mapsto[H_t]$, where
$H_t=\boldsymbol{H}|_{\p3\times\{t\}}$. In particular,
\begin{equation}\label{again E0 limit inst}
[H_0]\in\overline{\cali(3)}.
\end{equation}
In view of (\ref{bold calE}) the sheaf $H_0$ fits into the exact triple
$$ 0 \to \op3(-1) \stackrel{r}{\to} H_0 \to I_{Z_0,\p3}(1)\to 0; $$
composing the morphism $r$ in the previous equation with the standard monomorphism $H_0\to H_0^{\vee\vee}$
we obtain, using the triple (\ref{again embedded pt}), the following exact triples for $H_0^{\vee\vee}$:
\begin{equation}\label{E refl}
0 \to H_0 \to H_0^{\vee\vee} \to \calo_{q_1}\to 0, ~~{\rm and}
\end{equation}
\begin{equation}\label{E refl triple l=1}
0 \to \op3(-1) \overset{s}{\to} H_0^{\vee\vee} \to I_{(Z_0)_{red}/\p3}(1)\to 0.
\end{equation}
Now (\ref{Z0red l=1}) and (\ref{E refl triple l=1}) show that $s$ is a section of a reflexive sheaf $F_0^{\vee\vee}(1)$ having a disjoint union $(Z_0)_{red}$ of two lines and a reducible conic as its zero scheme. Hence, from the description of $\calr(0;3;2)$ given in item (i) above, it follows that the triple  (\ref{E refl triple l=1}) is a specialization of the triple (\ref{F,T1}) within a flat family of triples in which the nonsingular conic $C$ specializes into a reducible conic $\ell_3\cup \ell_4$, so that $Y$ specializes to $(Z_0)_{red}$. It follows that the triple (\ref{E refl}) is a flat specialization of the triple (\ref{R(0,3,l)}), so that
\begin{equation}\label{E0 limit T1}
[H_0]\in\overline{\calt(3,1)}.
\end{equation}
Finally, the case $l=1$ follows from (\ref{again E0 limit inst}) and (\ref{E0 limit T1}).

For $l=2,3,4$ the above argument goes through with the following modifications. 

For $l=2$, instead of (\ref{Z0red l=1}) one takes $(Z_0)_{red}=\ell_1\sqcup (\ell_2\cup\ell_3\cup \ell_4)$ with points $q_1=\ell_2\cap \ell_3$ and $q_2=\ell_3\cap \ell_4$, and consider it as a flat degeneration of a disjoint union of a line $\ell$ plus a smooth twisted cubic $C$ item (ii) above.

For $l=3$, one takes $(Z_0)_{red}=\ell_1\cup\ell_1\cup \ell_2\cup \ell_3)$ to be a chain of lines with three distinct points $q_1=\ell_1\cap \ell_2$, $q_2=\ell_2\cap \ell_3$, and $q_2=\ell_3\cap \ell_4$, considered as a flat degeneration of a nonsingular rational quartic curve $C$ from item (iii) above.

For $l=4$, one takes $(Z_0)_{red}=\ell_1\cup \ell_2\cup \ell_3\cup \ell_4$, to be a space union of lines with distinct intersection points $q_1=\ell_1\cap \ell_2$, $q_2=\ell_2\cap \ell_3$, $q_3=\ell_3\cap \ell_4$, $q_4=\ell_4\cap \ell_1$, considered as a flat degeneration of the nonsingular space elliptic quartic from item (iv) above.
\end{proof}

\begin{remark}\rm
Gruson and Trautmann conjectured that the boundary of charge 3 instanton bundles 
$$ \partial\cali(3) := \overline{\cali(3)}\setminus \cali(3) $$
has exactly 8 divisorial irreducible components, which can be divided into 2 types, cf. \cite[Remarque 3.6.8]{Per3}:
\begin{itemize}
\item[(I)] 4 components whose generic point corresponds to an instanton sheaf which is singular along a line, a smooth conic, a smooth twisted cubic or a smooth plane cubic; 
\item[(II)] 4 components whose generic point corresponds to a (non-instanton) sheaf which is singular along 2, 4, 6 or 8 points. 
\end{itemize}

The components of type (I) are, in the notation of \cite{JMT}, $\cald(m,3)$ for $m=1,2,3$ and $\overline{\calc(1,3,0)}\cap\overline{\cali(3)}$; all of these are known to be irreducible and divisorial, see \cite{GS}, \cite{JMT}, \cite[Th\'eor\`eme 0.1]{Per3}, and  \cite{Per4}. 

The components of type (II) are, in the notation of this paper, $\overline{\calt(3,l)}\cap\overline{\cali(3)}$ for $l=1,2,3,4$. Perrin showed in \cite{Per3} that $\overline{\calt(3,1)}\cap\overline{\cali(3)}$ is an irreducible divisor within ${\cali(3)}$.

Therefore, completing the proof of the Gruson--Trautmann conjecture amounts to showing that $\overline{\calt(3,l)}\cap\overline{\cali(3)}$ are irreducible and divisorial within ${\cali(3)}$ for $l=2,3,4$, and that $\partial\cali(3)$ has no other divisorial components.
\end{remark}

\begin{remark}\rm
It is interesting to note that the argument in Proposition \ref{T inter I} can be adapted to show that $\overline{\calt(3,1)}\cap\overline{\caln(0,1,2)}\ne\emptyset$. 

Indeed, consider a 1-dimensional flat family $\calz$ of curves in $\p3$ as in (\ref{calZ}), with base $U\ni0$ being an open subset of $\mathbb{A}^1$, such that
\begin{itemize}
\item [(a)] for $t\ne0$ the fibre $Z_t=\pi^{-1}(t)$ of the family $\calz$ is a disjoint union of a nonsingular elliptic cubic $C_t$ with a nonsingular space elliptic quartic $Q_t$;
\item [(b)] the zeroth fibre $Z_0$ of this family, being reduced, is the union a nonsingular elliptic cubic $C_0$ with a nonsingular space elliptic quartic $Q_0$ meeting a point $q$ 
$$ (Z_0)_{red} = C_0\cup Q_0,\ \ \ q := C_0\cap Q_0 = \{\mathrm{pt}\} $$
and as a scheme $Z_0$ has $q$ as an embedded point:
\begin{equation}\label{emb pt}
0\to \calo_{q} \to \calo_{Z_0} \to \calo_{(Z_0)_{red}} \to0.
\end{equation}
\end{itemize}
Next, consider the family $\boldsymbol{H}$ of sheaves on $\p3$ with base $U$ given by the exact triple
\begin{equation}\label{family U}
0\to \op3(-2)\boxtimes\calo_U \to \boldsymbol{H} \to
I_{\calz/\p3\times U}\otimes\calo_{\p3}(2)\boxtimes\calo_U \to 0.
\end{equation}
As before, set $H_t:=\boldsymbol{H}|_{\p3\times\{t\}}$. As observed by Hartshorne in \cite[Example 3.1.3]{H-vb}, the sheaves $H_t$ for $t\ne0$ are stable locally free sheaves with $c_1(H_t)=0$, $c_2(H_t)=3$ and $\alpha$ invariant equal to $1$; since instanton bundles have $\alpha=0$, it follows that $[H_t]\in\caln(0,1,2)$ when $t\ne0$. 

On the other hand, the reflexive sheaf $F$ defined by the exact triple
$$ 0 \to \op3(-2) \to F \to I_{(Z_0)_{red}/\p3}(2) \to 0 $$
yields a point in $\calr(0;3;2)$, since $(Z_0)_{red}$ is not contained in a hypersurface of degree 2, so that $F$ is $\mu$-stable; note, however, that the sheaves obtained in this way are not generic in $\calr(0;3;2)$, cf. \cite[proof of Lemma 3.6]{Chang}. Since $H_0$ fits into the exact triple
$$ 0 \to \op3(-2) \to H_0 \to I_{Z_0/\p3}(2) \to 0, $$
it then follows from (\ref{emb pt}) that  $H_0$ and $F$ are related via the exact triple
\begin{equation}\label{H0-F}
0 \to H_0 \to F \to \calo_{q} \to 0.
\end{equation}
The stability of $F$ implies that $H_0$ is also stable. Now, (\ref{family U}) imples that $[H_0]\in\overline{\caln(0,1,2)}$; on the other hand, (\ref{H0-F}) imples that $[H_0]\in\overline{\calt(3,1)}$, proving our claim.

We conjecture that $\overline{\calt(3,l)}\cap\overline{\caln(0,1,2)}\ne\emptyset$ also for $l=2,3,4$.
\end{remark}

\end{document}